\documentclass[11pt]{article}
\usepackage[utf8]{inputenc}

\usepackage[title]{appendix}
\usepackage{epsfig,epsf,fancybox}
\usepackage{amsmath}
\usepackage{mathrsfs}
\usepackage{amssymb}
\usepackage{graphicx}
\usepackage{color}
\usepackage{multirow}
\usepackage{paralist}
\usepackage{verbatim}
\usepackage{galois}
\usepackage{algorithm}
\usepackage{algorithmic}
\usepackage{boxedminipage}
\usepackage{booktabs}
\usepackage{accents}
\usepackage{stmaryrd}
\usepackage{subfig}
\usepackage{epstopdf}
\usepackage{amsthm}
\usepackage{cases}  
\usepackage[top=1in,bottom=1.2in,left=1in,right=1in,xetex]{geometry}
\usepackage[pdfborder={0 0 0},colorlinks=true,linkcolor=blue,CJKbookmarks=true]{hyperref}
\usepackage{enumerate}
\usepackage{amsrefs}

\newcommand\blfootnote[1]{%
	\begingroup
	\renewcommand\thefootnote{}\footnote{#1}%
	\addtocounter{footnote}{-1}%
	\endgroup
}

\allowdisplaybreaks[4]

\newtheorem{theorem}{Theorem}[section]
\newtheorem{lemma}[theorem]{Lemma}
\newtheorem{assumption}[theorem]{Assumption}
\newtheorem{proposition}[theorem]{Proposition}
\newtheorem{definition}[theorem]{Definition}
\newtheorem{remark}[theorem]{Remark}

\title{Mild Solution of Semilinear Rough Stochastic Evolution Equations\footnotemark[1]}
\author{Jiahao Liang\footnotemark[2] \ and Shanjian Tang\footnotemark[3]} 

\begin{document}

\maketitle

\pagenumbering{arabic}

\begin{abstract}
In this paper, we investigate a semilinear stochastic parabolic equation with a linear rough term 
$du_{t}=\left[L_{t}u_{t}+f\left(t, u_{t}\right)\right]dt+\left(G_{t}u_{t}+g_{t}\right)d\mathbf{X}_{t}+h\left(t, u_{t}\right)dW_{t}$, 
where $\left(L_{t}\right)_{t \in \left[0, T\right]}$ is a family of unbounded operators 
acting on a monotone family of interpolation Hilbert spaces, 
$\mathbf{X}$ is a two-step $\alpha$-Hölder rough path with $\alpha \in \left(1/3, 1/2\right]$ 
and $W$ is a Brownian motion. 
Existence and uniqueness of the mild solution are given through 
the stochastic controlled rough path approach and fixed-point argument. 
As a technical tool to define rough stochastic convolutions, 
we also develop a general mild stochastic sewing lemma, 
which is applicable for processes according to a monotone family. 
\end{abstract}

\renewcommand{\thefootnote}{\fnsymbol{footnote}}

\blfootnote{\textit{MSC2020 subject classifications}: 60H15, 60L50. }
\footnotetext[1]{This work was partially supported by National Natural Science Foundation of China (Grants No. 12031009),  Key Laboratory of Mathematics for Nonlinear Sciences (Ministry of Education), and Shanghai Key Laboratory for Contemporary Applied Mathematics, Fudan University, Shanghai 200433, China.}
\footnotetext[2]{School of Mathematical Sciences, Fudan University, Shanghai 200433, China. E-mail: jhliang20@fudan.edu.cn.}
\footnotetext[3]{Institute of Mathematical Finance and Department of Finance and Control Sciences, School of Mathematical Sciences, Fudan University, Shanghai 200433, P. R. China. Email: sjtang@fudan.edu.cn.}

\section{Introduction}
\label{Sec1}
Rough integrals $\int Yd\mathbf{X}$ and rough differential equations (RDEs) 
\begin{equation*}
	Y_{t}=Y_{0}+\int_{0}^{t} g\left(Y\right) d \mathbf{X}, \quad t \in \left[0, T\right], 
\end{equation*}
introduced by Lyons \cites{Lyo94, Lyo98} and Gubinelli \cite{Gub04}, 
for two-step $\alpha$-Hölder rough paths $\mathbf{X}=\left(X, \mathbb{X}\right)$ 
with $\alpha \in \left(\frac{1}{3}, \frac{1}{2}\right]$ and paths $Y$ controlled by $X$, 
extend the work of Young \cite{You36} and have received numerous attentions. 
On the one hand, PDEs driven by irregular paths have been well-studied. 
One of the important approaches is to study mild solutions of them. 
Semilinear Young PDEs have been studied through the mild approach by 
Gubinelli et al. \cite{GLT06} and Addona et al. \cites{ALT22, ALT24}. 
Mild solutions of semilinear rough PDEs were first studied by 
Gubinelli and Tindel \cite{GT10} in the case of polynomial nonlinearities. 
Local existence and uniqueness of mild solutions of general semilinear rough PDEs 
have been studied by Deya et al. \cite{DGT12}, Gerasimovičs and Hairer \cite{GH19}, Hesse and Neamţu \cites{HN19}
and Gerasimovičs et al. \cite{GHN21}. 
Deya et al. \cite{DGT12} and Hesse and Neamţu \cites{HN20, HesN22} give 
sufficient conditions for the existence of global mild solutions. 
Quasilinear rough PDEs have been studied by Hocquet and Neamţu \cite{HocN22}. 
On the other hand, differential equations driven simultaneously by Brownian motions 
and rough paths have also been well-studied. 
The well-posedness of (forward) rough SDEs has been established 
by Crisan et al. \cite{CDFO13}, Diehl et al. \cite{DOR15} and Friz et al. \cite{FHL21}. 
Rough BSDEs have been studied by Diehl and Friz \cite{DF12} and Liang and Tang \cite{LT23-1}.\\
\indent
In this paper, we consider the following semilinear stochastic partial differential equation with a linear rough term (rough SPDE)
\begin{equation} \label{rspdei}
	\left\{
		\begin{aligned}
			&du_{t}=\left[L_{t}u_{t}+f\left(t, u_{t}\right)\right]dt+\left(G_{t}u_{t}+g_{t}\right)d\mathbf{X}_{t}+h\left(t, u_{t}\right)dW_{t}, \quad t \in \left(0, T\right],\\
			&u_{0}=\xi,
		\end{aligned}
	\right.	
\end{equation}
where $\left(L_{t}\right)_{t \in \left[0, T\right]}$ is a family of unbounded operators 
acting on a monotone family of interpolation Hilbert spaces $\left(\mathcal{H}_{\gamma}\right)_{\gamma \in \mathbb{R}}$ 
and generating a propagator $\left(S_{s, t}\right)_{0 \leq s \leq t \leq T}$, 
$\mathbf{X}=\left(X, \mathbb{X}\right)$ is a two-step $\alpha$-Hölder rough path with $\alpha \in \left(\frac{1}{3}, \frac{1}{2}\right]$, 
$W$ is a standard Brownian motion, and coefficients $f, G, g$ and $h$ are random and time-varying. 
Such equations have been studied by Liang and Tang \cite{LT23-2} while $\mathbf{X}$ 
is a $\sigma$-Hölder continuous path with $\sigma \in \left(\frac{1}{2}, 1\right)$. 
It connects to the SPDE driven by $W$ and an independent fractional Brownian motion 
$B^{H}$ with Hurst parameter $H \in \left(\frac{1}{3}, \frac{1}{2}\right]$ 
\begin{equation*}
	\left\{
		\begin{aligned}
			&du_{t}=\left[L_{t}u_{t}+f\left(t, u_{t}\right)\right]dt+\left(G_{t}u_{t}+g_{t}\right)dB_{t}^{H}+h\left(t, u_{t}\right)dW_{t}, \quad t \in \left(0, T\right],\\
			&u_{0}=\xi.
		\end{aligned}
	\right.	
\end{equation*}
\indent
Through the propagator $S$, the rough SPDE \eqref{rspdei} can be formulated in a mild form
\begin{equation*}
	u_{t}=S_{0, t}\xi+\int_{0}^{t}S_{r, t}f\left(r, u_{r}\right)dr+\int_{0}^{t}S_{r, t}\left(G_{r}u_{r}+g_{r}\right)d\mathbf{X}_{r}+\int_{0}^{t}S_{r, t}h\left(r, u_{r}\right)dW_{r}, \quad t \in \left[0, T\right].
\end{equation*}
To solve this equation, we first need to establish rough stochastic convolutions 
\begin{equation} \label{rscidi}
	\int_{0}^{\cdot}S_{r, \cdot}Y_{r}d\mathbf{X}_{r}
\end{equation}
similar to classical rough convolutions, 
but for stochastic process pairs $\left(Y, Y^{\prime}\right)$ under certain conditions 
(called stochastic controlled rough paths according to the monotone family $\left(\mathcal{H}_{\gamma}\right)_{\gamma \in \mathbb{R}}$). 
To this end, we combine the stochastic sewing lemma 
introduced by Lê \cite[Theorem 2.1]{Le20} with the mild sewing lemma 
introduced by Gubinelli and Tindel \cite[Theorem 3.5]{GT10}, 
and develop a general mild stochastic sewing lemma. 
Through this lemma, rough stochastic convolutions \eqref{rscidi} can be well defined. 
Secondly, we need to find suitable conditions for a random operator family pair $\left(G, G^{\prime}\right)$ 
such that its composition with a stochastic controlled rough path $\left(Y, Y^{\prime}\right)$, 
i.e. $\left(GY, GY^{\prime}+G^{\prime}Y\right)$, is also a stochastic controlled rough paths. 
Such operator family pairs are called stochastic controlled operator families 
according to the monotone family $\left(\mathcal{H}_{\gamma}\right)_{\gamma \in \mathbb{R}}$. 
On these bases, by a fixed-point argument together with some estimates, 
we get the existence and uniqueness of mild solution of the rough SPDE \eqref{rspdei} under suitable conditions.
The continuity of the mild solution map and spatial regularity of the mild solution 
are also obtained. 
To our best knowledge, this is the first study to PDEs driven simultaneously by Brownian motions 
and rough paths.\\
\indent
The paper is organized as follows. 
Section \ref{Sec2} contains our fuctional analytic framework and some preliminary notations and results. 
In Section \ref{Sec3}, we introduce stochastic controlled rough paths and operator families 
according to a monotone family, and then investigate properties of convolutions and compositions of them. 
The mild stochastic sewing lemma is also developed. 
Mild solution of the rough SPDE \eqref{rspdei} is studied in Section \ref{Sec4}. 
In Section \ref{Sec5}, we provide a concrete example to illustrate our results. 

\section{Preliminaries}
\label{Sec2}
Throughout the paper, we fix $\alpha \in \left(\frac{1}{3}, \frac{1}{2}\right]$ and let 
$\beta \in \left(0, 1\right)$. 
Fixing a finite time horizon $T > 0$, 
let $\Delta_{2}:=\left\{\left(s, t\right): 0 \leq s \leq t \leq T\right\}$ 
and $\Delta_{3}:=\left\{\left(s, r, t\right): 0 \leq s \leq r \leq t \leq T\right\}$. 
For $\left(s, t\right) \in \Delta_{2}$, denote by $\mathcal{P}\left[s, t\right]$ 
the set of all partitions of the interval $\left[s, t\right]$, and by $\left|\pi\right|$ 
the mesh size of a partition $\pi \in \mathcal{P}\left[s, t\right]$. 
For Banach spaces $V$ and $\bar{V}$, define $\mathcal{L}\left(V, \bar{V}\right)$ as 
the space of bounded linear operators from $V$ to $\bar{V}$, endowed with the operator norm. 
For simplicity write $\mathcal{L}\left(V\right):=\mathcal{L}\left(V, V\right)$. 
Define $C\left(V, \bar{V}\right)$ as the space of bounded continuous maps 
from $V$ to $\bar{V}$, endowed with the maximum-norm.\\
\indent
Let $\left(\Omega, \mathcal{F}, \mathcal{F}_{t}, \mathbb{P}\right)$ 
be a filtered probability space satisfying the usual conditions and 
carrying a $d$-dimensional Brownian motion $W$. 
Denote by $\mathbb{E}_{t}$ the expectation operator conditioned at $\mathcal{F}_{t}$. 
For $m \in \left[2, \infty\right]$, a Banach space $\left(V, \left|\cdot\right|_{V}\right)$ and 
a subfield $\mathcal{G}$ of $\mathcal{F}$, 
define $L^{m}\left(\Omega, \mathcal{G}, V\right)$ as 
the space of $\mathcal{G}$-measurable $L^{m}$-integrable 
$V$-valued random variables $\xi$, endowed with the norm 
$\left\|\xi\right\|_{m, V}:=\left\|\left|\xi\right|_{V}\right\|_{m}$. 
For simplicity write $L^{m}\left(\Omega, V\right):=L^{m}\left(\Omega, \mathcal{F}, V\right)$.\\
\indent
Throughout the paper, we write $a \lesssim b$ provided there exists constant $C > 0$ such that $a \leq Cb$. 
Unless stated otherwise, the constant $C$ only depends on parameters 
such that $\alpha, \beta, m$, etc, and $T \vee 1$. 
For fuction pairs $\left(f, \bar{f}\right)$, we simply write $\Delta f:=f-\bar{f}$. 
\subsection{Functional analytic framework}
To avoid issues related to stochastic integrals, 
we work with Hilbert spaces rather than Banach spaces as in \cites{GHN21, HesN22}. 
Let $\left(\mathcal{H}_{\gamma}, \left|\cdot\right|_{\gamma}\right)_{\gamma \in \mathbb{R}}$ 
be a monotone family of interpolation Hilbert spaces, i.e. 
$\mathcal{H}_{\gamma}$ is a separable Hilbert space for $\gamma \in \mathbb{R}$ 
such that for $\gamma_{1} \leq \gamma_{2}$, 
$\mathcal{H}_{\gamma_{2}}$ is a continuously embedded, dense subspace of $\mathcal{H}_{\gamma_{1}}$, 
and the following interpolation inequality holds 
\begin{equation} \label{ii}
	\left|u\right|_{\gamma_{2}}^{\gamma_{3}-\gamma_{1}} \lesssim \left|u\right|_{\gamma_{1}}^{\gamma_{3}-\gamma_{2}}\left|u\right|_{\gamma_{3}}^{\gamma_{2}-\gamma_{1}}, \quad \forall \gamma_{1} \leq \gamma_{2} \leq \gamma_{3}, \quad \forall u \in \mathcal{H}_{\gamma_{3}}.
\end{equation}
Let $\left(L_{t}\right)_{t \in \left[0, T\right]}$ be a family of unbounded operators 
such that for every $t \in \left[0, T\right]$ and $\gamma \in \mathbb{R}$, 
$L_{t}$ is a closed densely defined linear operator on $\mathcal{H}_{\gamma}$ 
and its domain contains $\mathcal{H}_{\gamma+1}$. 
Assume $\left(L_{t}\right)_{t \in \left[0, T\right]}$ generates a propagator 
$S: \Delta_{2} \rightarrow \cap_{\gamma \in \mathbb{R}}\mathcal{L}\left(\mathcal{H}_{\gamma}\right)$ 
(see \cite[Definitions 2.7 and 2.11]{GHN21} for the precise definition). 
Indeed, $S$ is a strongly continuous operator family satisfying 
$S_{r, t}S_{s, r}=S_{s, t}$ for every $\left(s, r, t\right) \in \Delta_{3}$ 
and $S_{t, t}=id$ for every $t \in \left[0, T\right]$. 
By interpolation we have the following usefull result. 
\begin{proposition} \label{asp}
	For $\gamma_{1} \leq \gamma_{2} \leq \gamma_{1}+1$, we have 
	\begin{equation*}
		\left|S_{s, t}u\right|_{\gamma_{2}} \lesssim \left|t-s\right|^{\gamma_{1}-\gamma_{2}}\left|u\right|_{\gamma_{1}}, \quad \left|S_{s, t}u-u\right|_{\gamma_{1}} \lesssim \left|t-s\right|^{\gamma_{2}-\gamma_{1}}\left|u\right|_{\gamma_{2}}, \quad \forall 0 \leq s < t \leq T.
	\end{equation*}
\end{proposition}
In the sequel of this paper, we let $\gamma, \gamma_{1}$ and $\gamma_{2}$ be any real numbers 
and $m \in \left[2, \infty\right)$. 
For simplicity we write $\mathcal{H}_{\gamma}^{e}:=\mathcal{L}\left(\mathbb{R}^{e}, \mathcal{H}_{\gamma}\right)$ 
and also denote by $\left|\cdot\right|_{\gamma}$ its norm. 
For $f \in \mathcal{L}\left(\mathcal{H}_{\gamma_{1}}^{e_{1}}, \mathcal{H}_{\gamma_{2}}^{e_{2}}\right)$, 
define 
\begin{equation*}
	\left|f\right|_{\left(\gamma_{1}, \gamma_{2}\right)\text{-}op}:=\sup_{\left|u\right|_{\gamma_{1}} \leq 1}\left|fu\right|_{\gamma_{2}}.
\end{equation*}
Write 
$\left\|\cdot\right\|_{m, \gamma}:=\left\|\cdot\right\|_{m, \mathcal{H}_{\gamma}^{e_{1}}}$ 
and $\left\|\cdot\right\|_{\infty, \left(\gamma_{1}, \gamma_{2}\right)\text{-}op}:=\left\|\cdot\right\|_{\infty, \mathcal{L}\left(\mathcal{H}_{\gamma_{1}}^{e_{1}}, \mathcal{H}_{\gamma_{2}}^{e_{2}}\right)}$. 
\subsection{Increment operators and Hölder type spaces}
For $Y: \left[0, T\right] \times \Omega \rightarrow \mathcal{H}_{\gamma}^{e_{1}}$, 
define $\delta Y, \hat{\delta}Y: \Delta_{2} \times \Omega \rightarrow \mathcal{H}_{\gamma}^{e_{1}}$ 
as the increment and mild increment of $Y$ respectively, i.e. 
\begin{equation*}
	\delta Y_{s, t}:=Y_{t}-Y_{s}, \quad \hat{\delta} Y_{s, t}:=Y_{t}-S_{s, t}Y_{s}, \quad \forall \left(s, t\right) \in \Delta_{2}.
\end{equation*}
Similarly, for $A: \Delta_{2} \times \Omega \rightarrow \mathcal{H}_{\gamma}^{e_{2}}$, 
define $\delta A, \hat{\delta}A: \Delta_{3} \times \Omega \rightarrow \mathcal{H}_{\gamma}^{e_{2}}$ by 
\begin{equation*}
	\delta A_{s, r, t}:=A_{s, t}-A_{s, r}-A_{r, t}, \quad \hat{\delta} A_{s, r, t}:=A_{s, t}-S_{r, t}A_{s, r}-A_{r, t}, \quad \forall \left(s, r, t\right) \in \Delta_{3}.
\end{equation*}
$A$ is said to be adapted if $A_{s, t}$ is $\mathcal{F}_{t}$-measurable for every $\left(s, t\right) \in \Delta_{2}$. 
Denote by $\mathbb{E}_{\cdot}A$ the process: $\left(s, t\right) \mapsto \mathbb{E}_{s}A_{s, t}, \left(s, t\right) \in \Delta$. 
Define $C_{2}^{\beta}L_{m}\mathcal{H}_{\gamma}^{e_{2}}$ as the space of processes 
$A \in C\left(\Delta_{2}, L^{m}\left(\Omega, \mathcal{H}_{\gamma}^{e_{2}}\right)\right)$ 
measurable and adapted such that 
\begin{equation*}
	\left\|A\right\|_{\beta, m, \gamma}:=\sup_{0 \leq s < t \leq T}\frac{\left\|A_{s, t}\right\|_{m, \gamma}}{|t-s|^{\beta}} < \infty. 
\end{equation*}
Note that $A \in C_{2}^{\beta}L_{m}\mathcal{H}_{\gamma}^{e_{2}}$ implies $A_{t, t}=0$ 
for every $t \in \left[0, T\right]$. 
For simplicity, write $CL_{m}\mathcal{H}_{\gamma}^{e_{1}}:=C\left(\left[0, T\right], L^{m}\left(\Omega, \mathcal{H}_{\gamma}^{e_{1}}\right)\right)$ 
and 
\begin{equation*}
	\left\|Y\right\|_{0, m, \gamma}:=\sup_{t \in \left[0, T\right]}\left\|Y_{t}\right\|_{m, \gamma}. 
\end{equation*}
Define $C^{\beta}L_{m}\mathcal{H}_{\gamma}^{e_{1}}$ (resp. $\hat{C}^{\beta}L_{m}\mathcal{H}_{\gamma}^{e_{1}}$) 
as the space of processes $Y \in CL_{m}\mathcal{H}_{\gamma}^{e_{1}}$ 
measurable and adapted such that 
$\left\|\delta Y\right\|_{\beta, m, \gamma}$ (resp. $\left\|\hat{\delta} Y\right\|_{\beta, m, \gamma}$) is finite. 
Then define $E^{\beta}L_{m}\mathcal{H}_{\gamma}^{e_{1}}:=CL_{m}\mathcal{H}_{\gamma}^{e_{1}} \cap C^{\beta}L_{m}\mathcal{H}_{\gamma-\beta}^{e_{1}}$, 
endowed with the norm 
\begin{equation*}
	\left\|Y\right\|_{E^{\beta}L_{m}\mathcal{H}_{\gamma}}:=\left\|Y\right\|_{0, m, \gamma}+\left\|\delta Y\right\|_{\beta, m, \gamma-\beta}.
\end{equation*}
\begin{proposition} \label{in}
	For $\theta \in \left(0, \beta\right)$, we have 
	$E^{\beta}L_{m}\mathcal{H}_{\gamma}^{e_{1}} \subset E^{\theta}L_{m}\mathcal{H}_{\gamma}^{e_{1}}$ and 
	\begin{equation} \label{ini}
		\left\|Y\right\|_{E^{\theta}L_{m}\mathcal{H}_{\gamma}} \lesssim \left\|Y\right\|_{E^{\beta}L_{m}\mathcal{H}_{\gamma}}. 
	\end{equation}
\end{proposition}
\begin{proof}
	For every $Y \in E^{\beta}L_{m}\mathcal{H}_{\gamma}^{e_{1}}$, in view of \eqref{ii} 
	and applying Hölder's inequality, we have 
	\begin{align*}
		\left\|\delta Y_{s, t}\right\|_{m, \gamma-\theta} &\lesssim \left(\mathbb{E}\left|\delta Y_{s, t}\right|_{\gamma-\beta}^{\frac{\theta m}{\beta}}\left|\delta Y_{s, t}\right|_{\gamma}^{\frac{\left(\beta-\theta\right)m}{\beta}}\right)^{\frac{1}{m}} \lesssim \left\|\delta Y_{s, t}\right\|_{m, \gamma-\beta}^{\frac{\theta}{\beta}}\left\|\delta Y_{s, t}\right\|_{m, \gamma}^{\frac{\beta-\theta}{\beta}}\\
		&\lesssim \left|t-s\right|^{\theta}\left\|\delta Y\right\|_{\beta, m, \gamma-\beta}^{\frac{\theta}{\beta}}\left\|Y\right\|_{0, m, \gamma}^{\frac{\beta-\theta}{\beta}} \lesssim \left|t-s\right|^{\theta}\left\|Y\right\|_{E^{\beta}L_{m}\mathcal{H}_{\gamma}}, \quad \forall \left(s, t\right) \in \Delta_{2},
	\end{align*}
	which gives $Y \in C^{\theta}L_{m}\mathcal{H}_{\gamma-\theta}^{e_{1}}$ and 
	\begin{equation*}
		\left\|\delta Y\right\|_{\theta, m, \gamma-\theta} \lesssim \left\|Y\right\|_{E^{\beta}L_{m}\mathcal{H}_{\gamma}}.
	\end{equation*}
	Hence, $Y \in E^{\theta}L_{m}\mathcal{H}_{\gamma}^{e_{1}}$ and the estimate \eqref{ini} holds. 
\end{proof}
The following result can be found in \cite[Proposition 2.2]{LT23-2}.
\begin{proposition} \label{ne}
	$E^{\beta}L_{m}\mathcal{H}_{\gamma}^{e_{1}}=CL_{m}\mathcal{H}_{\gamma}^{e_{1}} \cap \hat{C}^{\beta}L_{m}\mathcal{H}_{\gamma-\beta}^{e_{1}}$ 
	and we have 
	\begin{equation*}
		\left\|Y\right\|_{E^{\beta}L_{m}\mathcal{H}_{\gamma}} \lesssim \left\|Y\right\|_{0, m, \gamma}+\left\|\hat{\delta} Y\right\|_{\beta, m, \gamma-\beta} \lesssim \left\|Y\right\|_{E^{\beta}L_{m}\mathcal{H}_{\gamma}}. 
	\end{equation*}
\end{proposition}
Similarly, define 
$C_{2}^{\beta}L_{\infty}\mathcal{L}\left(\mathcal{H}_{\gamma_{1}}^{e_{1}}, \mathcal{H}_{\gamma_{2}}^{e_{2}}\right)$ as the space of measurable adapted processes 
$g: \Delta_{2} \times \Omega \rightarrow \mathcal{L}\left(\mathcal{H}_{\gamma_{1}}^{e_{1}}, \mathcal{H}_{\gamma_{2}}^{e_{2}}\right)$ such that 
$g \in C\left(\Delta_{2}, L^{\infty}\left(\Omega, \mathcal{L}\left(\mathcal{H}_{\gamma_{1}}^{e_{1}}, \mathcal{H}_{\gamma_{2}}^{e_{2}}\right)\right)\right)$ 
and 
\begin{equation*}
	\left\|g\right\|_{\beta, \infty, \left(\gamma_{1}, \gamma_{2}\right)\text{-}op}:=\sup_{0 \leq s < t \leq T}\frac{\left\|g_{s, t}\right\|_{\infty, \left(\gamma_{1}, \gamma_{2}\right)\text{-}op}}{|t-s|^{\beta}} < \infty.
\end{equation*}
Write $CL_{\infty}\mathcal{L}\left(\mathcal{H}_{\gamma_{1}}^{e_{1}}, \mathcal{H}_{\gamma_{2}}^{e_{2}}\right):=C\left(\left[0, T\right], L^{\infty}\left(\Omega, \mathcal{L}\left(\mathcal{H}_{\gamma_{1}}^{e_{1}}, \mathcal{H}_{\gamma_{2}}^{e_{2}}\right)\right)\right)$ 
and 
\begin{equation*}
	\left\|f\right\|_{0, \infty, \left(\gamma_{1}, \gamma_{2}\right)\text{-}op}:=\sup_{t \in \left[0, T\right]}\left\|f_{t}\right\|_{\infty, \left(\gamma_{1}, \gamma_{2}\right)\text{-}op}.
\end{equation*}
Define $C^{\beta}L_{\infty}\mathcal{L}\left(\mathcal{H}_{\gamma_{1}}^{e_{1}}, \mathcal{H}_{\gamma_{2}}^{e_{2}}\right)$ 
as the space of processes $f \in CL_{\infty}\mathcal{L}\left(\mathcal{H}_{\gamma_{1}}^{e_{1}}, \mathcal{H}_{\gamma_{2}}^{e_{2}}\right)$ 
measurable and adapted such that 
$\left\|\delta f\right\|_{\beta, m, \left(\gamma_{1}, \gamma_{2}\right)\text{-}op}$ is finite. 
Then define $E^{\beta}L_{\infty}\mathcal{L}_{\gamma_{1}, \gamma_{2}}\left(\mathcal{H}^{e_{1}}, \mathcal{H}^{e_{2}}\right):=CL_{\infty}\mathcal{L}\left(\mathcal{H}_{\gamma_{1}}^{e_{1}}, \mathcal{H}_{\gamma_{2}}^{e_{2}}\right) \cap C^{\beta}L_{\infty}\mathcal{L}\left(\mathcal{H}_{\gamma_{1}-\beta}^{e_{1}}, \mathcal{H}_{\gamma_{2}-\beta}^{e_{2}}\right)$, 
endowed with the norm 
\begin{equation*}
	\left\|f\right\|_{E^{\beta}L_{\infty}\mathcal{L}_{\gamma_{1}, \gamma_{2}}}:=\left\|f\right\|_{0, \infty, \left(\gamma_{1}, \gamma_{2}\right)\text{-}op}+\left\|f\right\|_{0, \infty, \left(\gamma_{1}-\beta, \gamma_{2}-\beta\right)\text{-}op}+\left\|\delta f\right\|_{\beta, \infty, \left(\gamma_{1}-\beta, \gamma_{2}-\beta\right)\text{-}op}.
\end{equation*}
In general, 
$f \in E^{\beta}L_{\infty}\mathcal{L}_{\gamma_{1}, \gamma_{2}}\left(\mathcal{H}^{e_{1}}, \mathcal{H}^{e_{2}}\right)$ 
cannot imply $f \in E^{\theta}L_{\infty}\mathcal{L}_{\gamma_{1}, \gamma_{2}}\left(\mathcal{H}^{e_{1}}, \mathcal{H}^{e_{2}}\right)$ 
for $\theta \in \left(0, \beta\right)$, since $f$ may not take values in 
$\mathcal{L}\left(\mathcal{H}_{\gamma_{1}-\theta}^{e_{1}}, \mathcal{H}_{\gamma_{2}-\theta}^{e_{2}}\right)$. 
The following result can be found in \cite[Proposition 4.3]{LT23-2}. 
\begin{proposition} \label{lcp}
	Let $Y \in E^{\beta}L_{m}\mathcal{H}_{\gamma_{1}}^{e_{1}}$ 
	and $f \in E^{\beta}L_{\infty}\mathcal{L}_{\gamma_{1}, \gamma_{2}}\left(\mathcal{H}^{e_{1}}, \mathcal{H}^{e_{2}}\right)$. 
	Then we have $fY \in E^{\beta}L_{m}\mathcal{H}_{\gamma_{2}}^{e_{2}}$ and 
	\begin{equation*}
		\left\|fY\right\|_{E^{\beta}L_{m}\mathcal{H}_{\gamma_{2}}} \lesssim \left\|f\right\|_{E^{\beta}L_{\infty}\mathcal{L}_{\gamma_{1}, \gamma_{2}}}\left\|Y\right\|_{E^{\beta}L_{m}\mathcal{H}_{\gamma_{1}}}.
	\end{equation*}
\end{proposition}
\subsection{Rough paths}
For $X: \left[0, T\right] \rightarrow \mathbb{R}^{e_{1}}$ and $A: \Delta_{2} \rightarrow \mathbb{R}^{e_{2}}$, 
we can define $\delta X$ and $\delta A$ in a similar way as the increment of $X$ and $A$, respectively. 
Define $C^{\beta}\left(\left[0, T\right], \mathbb{R}^{e_{1}}\right)$ (resp. $C_{2}^{\beta}\left(\left[0, T\right], \mathbb{R}^{e_{2}}\right)$) 
as the space of continuous fuctions $X: \left[0, T\right] \rightarrow \mathbb{R}^{e_{1}}$ 
(resp. $A: \Delta_{2} \rightarrow \mathbb{R}^{e_{2}}$) such that 
$\left|\delta X\right|_{\beta}$ (resp. $\left|A\right|_{\beta}$) is finite, where 
\begin{equation*}
	\left|A\right|_{\beta}:=\sup_{0 \leq s < t \leq T}\frac{\left|A_{s, t}\right|}{|t-s|^{\beta}}. 
\end{equation*}
\begin{definition}
	We call $\mathbf{X}=\left(X, \mathbb{X}\right)$ a two-step $\alpha$-Hölder rough path 
	over $\mathbb{R}^{e}$, denoted by 
	$\mathbf{X} \in \mathscr{C}^{\alpha}\left([0, T], \mathbb{R}^{e}\right)$, 
	if $X \in C^{\alpha}\left(\left[0, T\right], \mathbb{R}^{e}\right)$ and 
	$\mathbb{X} \in C_{2}^{2\alpha}\left(\left[0, T\right], \mathbb{R}^{e \times e}\right)$ 
	satisfy the Chen's relation 
    \begin{equation} \label{Chen}
      \delta \mathbb{X}_{s, r, t}=\delta X_{s, r} \otimes \delta X_{r, t}, \quad \forall \left(s, r, t\right) \in \Delta_{2}.
    \end{equation}
\end{definition}
The rough path space $\mathscr{C}^{\alpha}$ is equipped 
with the pseudometric 
\begin{equation*}
	\rho_{\alpha}\left(\mathbf{X}, \bar{\mathbf{X}}\right)=\left|\mathbf{X}-\bar{\mathbf{X}}\right|_{\alpha}:=\left|\delta X-\delta \bar{X}\right|_{\alpha}+\left|\mathbb{X}-\bar{\mathbb{X}}\right|_{2\alpha}. 
\end{equation*}
Since the initial value $X_{0}$ of $X$ has no effect on integrals against $\mathbf{X}$, 
we always assume without loss of generality that $X_{0}=0$. 
Then $\rho_{\alpha}$ is a true metric on the subspace 
$\left\{\mathbf{X} \in \mathscr{C}^{\alpha}: X_{0}=0\right\}$. 

\section{Convolution and composition}
\label{Sec3}
In this section, we build a framework for studying the rough stochastic evolution equation \eqref{rspdei}. 
To this end, we will introduce stochastic controlled rough paths 
according to the monotone family $\left(\mathcal{H}_{\gamma}\right)_{\gamma \in \mathbb{R}}$, 
and establish rough stochastic convolutions of them 
against rough paths $\mathbf{X} \in \mathscr{C}^{\alpha}$. 
We will also introduce stochastic controlled operator families according to 
$\left(\mathcal{H}_{\gamma}\right)_{\gamma \in \mathbb{R}}$ 
and investigate their compositions with stochastic controlled rough paths. 
Unless stated otherwise, we fix $0 < \beta^{\prime} \leq \beta \leq \alpha$. 
\subsection{Stochastic controlled rough paths according to a monotone family}
We first introduce the notion of stochastic controlled rough paths according to a monotone family, 
which combines stochastic controlled rough paths introduced in \cite[Definitions 3.2 and 3.4]{FHL21} 
and controlled paths according to a monotone family introduced in \cite[Definition 4.3]{GHN21}. 
\begin{definition} \label{scrp}
	Given $X \in C^{\alpha}\left(\left[0, T\right], \mathbb{R}^{e}\right)$, 
	we call $\left(Y, Y^{\prime}\right)$ an $L^{m}$-integrable 
	$\left(\beta, \beta^{\prime}\right)$-Hölder 
	stochastic controlled rough path (controlled by $X$) according to $\mathcal{H}_{\gamma}^{e_{1}}$, 
	denoted by $\left(Y, Y^{\prime}\right) \in \mathbf{D}_{X}^{\beta, \beta^{\prime}}L_{m}\mathcal{H}_{\gamma}^{e_{1}}$, 
	if $Y \in E^{\beta}L_{m}\mathcal{H}_{\gamma}^{e_{1}}$, 
	$Y^{\prime} \in E^{\beta^{\prime}}L_{m}\mathcal{H}_{\gamma-\beta}^{e_{1} \times e}$ and 
	$\mathbb{E}_{\cdot}R^{Y} \in C_{2}^{\beta+\beta^{\prime}}L_{m}\mathcal{H}_{\gamma-\beta-\beta^{\prime}}^{e_{1}}$, 
	where the remainder $R^{Y}$ is defined by 
	\begin{equation*}
		R_{s, t}^{Y}:=\delta Y_{s, t}-Y_{s}^{\prime}\delta X_{s, t}, \quad \forall \left(s, t\right) \in \Delta_{2}.
	\end{equation*}
	Moreover, we write $\left(Y, Y^{\prime}\right) \in \mathcal{D}_{X}^{\beta, \beta^{\prime}}L_{m}\mathcal{H}_{\gamma}^{e_{1}}$ 
	if additionally $R^{Y} \in C_{2}^{\beta+\beta^{\prime}}L_{m}\mathcal{H}_{\gamma-\beta-\beta^{\prime}}^{e_{1}}$. 
\end{definition}
Equipped with the norm 
\begin{equation*}
	\left\|Y, Y^{\prime}\right\|_{\mathbf{D}_{X}^{\beta, \beta^{\prime}}L_{m}\mathcal{H}_{\gamma}}:=\left\|Y\right\|_{E^{\beta}L_{m}\mathcal{H}_{\gamma}}+\left\|Y^{\prime}\right\|_{E^{\beta^{\prime}}L_{m}\mathcal{H}_{\gamma-\beta}}+\left\|\mathbb{E}_{\cdot}R^{Y}\right\|_{\beta+\beta^{\prime}, m, \gamma-\beta-\beta^{\prime}},
\end{equation*}
$\mathbf{D}_{X}^{\beta, \beta^{\prime}}L_{m}\mathcal{H}_{\gamma}^{e_{1}}$ is a Banach space 
and $\mathcal{D}_{X}^{\beta, \beta^{\prime}}L_{m}\mathcal{H}_{\gamma}^{e_{1}}$ is its Banach subspace. 
To indicate the underlying time inteval $\left[0, T\right]$, 
we use notations $\mathbf{D}_{X}^{\beta, \beta^{\prime}}L_{m}\mathcal{H}_{\gamma}^{e_{1}}\left[0, T\right]$ and 
$\left\|\cdot, \cdot\right\|_{\mathbf{D}_{X}^{\beta, \beta^{\prime}}L_{m}\mathcal{H}_{\gamma}\left[0, T\right]}$. 
The following result gives an equivalent definition of stochastic controlled rough paths according to a monotone family. 
\begin{proposition} \label{de}
	$\left(Y, Y^{\prime}\right) \in \mathbf{D}_{X}^{\beta, \beta^{\prime}}L_{m}\mathcal{H}_{\gamma}^{e_{1}}$ 
	if and only if $Y \in E^{\beta}L_{m}\mathcal{H}_{\gamma}^{e_{1}}$, 
	$Y^{\prime} \in E^{\beta^{\prime}}L_{m}\mathcal{H}_{\gamma-\beta}^{e_{1} \times e}$ and 
	$\mathbb{E}_{\cdot}\hat{R}^{Y} \in C_{2}^{\beta+\beta^{\prime}}L_{m}\mathcal{H}_{\gamma-\beta-\beta^{\prime}}^{e_{1}}$, 
	where the mild remainder $\hat{R}^{Y}$ is defined by 
	\begin{equation*}
		\hat{R}_{s, t}^{Y}:=\hat{\delta} Y_{s, t}-S_{s, t}Y_{s}^{\prime}\delta X_{s, t}, \quad \forall \left(s, t\right) \in \Delta_{2}.
	\end{equation*}
	Similarly, $\left(Y, Y^{\prime}\right) \in \mathcal{D}_{X}^{\beta, \beta^{\prime}}L_{m}\mathcal{H}_{\gamma}^{e_{1}}$ 
	if and only if $Y \in E^{\beta}L_{m}\mathcal{H}_{\gamma}^{e_{1}}$, 
	$Y^{\prime} \in E^{\beta^{\prime}}L_{m}\mathcal{H}_{\gamma-\beta}^{e_{1} \times e}$ and 
	$\hat{R}^{Y} \in C_{2}^{\beta+\beta^{\prime}}L_{m}\mathcal{H}_{\gamma-\beta-\beta^{\prime}}^{e_{1}}$. 
	Moreover, we have 
	\begin{align*}
		\left\|Y, Y^{\prime}\right\|_{\mathbf{D}_{X}^{\beta, \beta^{\prime}}L_{m}\mathcal{H}_{\gamma}}\left(1+\left|\mathbf{X}\right|_{\alpha}\right)^{-1} &\lesssim \left\|Y\right\|_{E^{\beta}L_{m}\mathcal{H}_{\gamma}}+\left\|Y^{\prime}\right\|_{E^{\beta^{\prime}}L_{m}\mathcal{H}_{\gamma-\beta}}+\left\|\mathbb{E}_{\cdot}\hat{R}^{Y}\right\|_{\beta+\beta^{\prime}, m, \gamma-\beta-\beta^{\prime}}\\
		&\lesssim \left\|Y, Y^{\prime}\right\|_{\mathbf{D}_{X}^{\beta, \beta^{\prime}}L_{m}\mathcal{H}_{\gamma}}\left(1+\left|\mathbf{X}\right|_{\alpha}\right). 
	\end{align*}
\end{proposition}
\begin{proof}
	It suffices to show that 
	\begin{equation} \label{dei}
		\left\|R^{Y}-\hat{R}^{Y}\right\|_{\beta+\beta^{\prime}, m, \gamma-\beta-\beta^{\prime}} \lesssim \left\|Y\right\|_{E^{\beta}L_{m}\mathcal{H}_{\gamma}}+\left\|Y^{\prime}\right\|_{E^{\beta^{\prime}}L_{m}\mathcal{H}_{\gamma-\beta}}\left|\mathbf{X}\right|_{\alpha},
	\end{equation}
	for every $\left(Y, Y^{\prime}\right) \in E^{\beta}L_{m}\mathcal{H}_{\gamma}^{e_{1}} \times E^{\beta^{\prime}}L_{m}\mathcal{H}_{\gamma-\beta}^{e_{1} \times e}$. 
	Indeed, by Proposition \ref{asp} we have 
	\begin{align*}
		\left\|R_{s, t}^{Y}-\hat{R}_{s, t}^{Y}\right\|_{m, \gamma-\beta-\beta^{\prime}}&=\left\|\left(S_{s, t}-id\right)Y_{s}+\left(S_{s, t}-id\right)Y_{s}^{\prime}\delta X_{s, t}\right\|_{m, \gamma-\beta-\beta^{\prime}}\\
		&\lesssim \left|t-s\right|^{\beta+\beta^{\prime}}\left\|Y_{s}\right\|_{m, \gamma}+\left|t-s\right|^{\beta^{\prime}}\left\|Y_{s}^{\prime}\right\|_{m, \gamma-\beta}\left|\delta X_{s, t}\right|\\
		&\lesssim \left(\left\|Y\right\|_{0, m, \gamma}+T^{\alpha-\beta}\left\|Y^{\prime}\right\|_{0, m, \gamma-\beta}\left|\mathbf{X}\right|_{\alpha}\right)\left|t-s\right|^{\beta+\beta^{\prime}}, \quad \forall \left(s, t\right) \in \Delta_{2}.
	\end{align*}
	Hence, the estimate \eqref{dei} holds. 
\end{proof}
For $X, \bar{X} \in C^{\alpha}\left(\left[0, T\right], \mathbb{R}^{e}\right)$, 
$\left(Y, Y^{\prime}\right) \in \mathbf{D}_{X}^{\beta, \beta^{\prime}}L_{m}\mathcal{H}_{\gamma}^{e_{1}}$ 
and $\left(\bar{Y}, \bar{Y}^{\prime}\right) \in \mathbf{D}_{\bar{X}}^{\beta, \beta^{\prime}}L_{m}\mathcal{H}_{\gamma}^{e_{1}}$, 
we introduce the ``distance" 
\begin{equation*}
	\left\|Y, Y^{\prime}; \bar{Y}, \bar{Y}^{\prime}\right\|_{\mathbf{D}_{X, \bar{X}}^{\beta, \beta^{\prime}}L_{m}\mathcal{H}_{\gamma}}:=\left\|\Delta Y\right\|_{E^{\beta}L_{m}\mathcal{H}_{\gamma}}+\left\|\Delta Y^{\prime}\right\|_{E^{\beta^{\prime}}L_{m}\mathcal{H}_{\gamma-\beta}}+\left\|\mathbb{E}_{\cdot}\Delta R^{Y}\right\|_{\beta+\beta^{\prime}, m, \gamma-\beta-\beta^{\prime}}.
\end{equation*}
Although this ``distance" is not a metric for $X \neq \bar{X}$, since 
$\left(Y, Y^{\prime}\right)$ and $\left(\bar{Y}, \bar{Y}^{\prime}\right)$ 
lie in different Banach spaces, 
it will be usefull to describe the stability of convolution and composition 
and continuity of the mild solution map. 
Analogous to the proof of estimate \eqref{dei}, we can show that 
\begin{align}
	\left\|\Delta R^{Y}-\Delta \hat{R}^{Y}\right\|_{\beta+\beta^{\prime}, m, \gamma-\beta-\beta^{\prime}} &\lesssim \left\|Y^{\prime}\right\|_{E^{\beta^{\prime}}L_{m}\mathcal{H}_{\gamma-\beta}}\rho_{\alpha}\left(\mathbf{X}, \bar{\mathbf{X}}\right) \notag\\
	&\quad+\left\|\Delta Y\right\|_{E^{\beta}L_{m}\mathcal{H}_{\gamma}}+\left\|\Delta Y^{\prime}\right\|_{E^{\beta^{\prime}}L_{m}\mathcal{H}_{\gamma-\beta}}\left|\bar{\mathbf{X}}\right|_{\alpha}. \label{des}
\end{align}
\subsection{Rough stochastic convolutions}
To establish rough stochastic convolutions, we need the following 
mild stochastic sewing lemma, which combines the stochastic sewing lemma 
introduced by Lê \cite[Theorem 2.1]{Le20} and \cite[Theorem 3.1]{Le23} in general Banach spaces, 
and the mild sewing lemma introduced by Gubinelli and Tindel \cite[Theorem 3.5]{GT10}. 
See Li and Sieber \cite[Proposition 3.1]{LS22} for another version. 
\begin{lemma} \label{mssl}
	Let $A \in C_{2}^{\alpha}L_{m}\mathcal{H}_{\gamma}$. 
	Assume there exist positive constants $K_{1}, K_{2}, z_{1}, z_{2}$ and a process 
	$\Lambda: \Delta_{3} \times \Omega \rightarrow \mathcal{H}_{\gamma}$ such that 
	$\frac{1}{2} < z_{1} \leq 1 < z_{2}$ and 
	\begin{equation*}
		\hat{\delta} A_{s, r, t}=S_{r, t}\Lambda_{s, r, t},
	\end{equation*}
	\begin{equation*}
		\left\|\Lambda_{s, r, t}\right\|_{m, \gamma+z_{2}-z_{1}} \leq K_{1}\left|t-s\right|^{\frac{1}{2}}\left|t-r\right|^{z_{1}-\frac{1}{2}}, \quad \left\|\mathbb{E}_{s}\Lambda_{s, r, t}\right\|_{m, \gamma} \leq K_{2}\left|t-s\right|\left|t-r\right|^{z_{2}-1},
	\end{equation*}
	for every $\left(s, r, t\right) \in \Delta_{3}$. 
	Then there exists unique $\mathcal{A} \in \hat{C}^{\alpha}L_{m}\mathcal{H}_{\gamma}$ with $\mathcal{A}_{0}=0$ such that 
	\begin{equation*}
		\lim_{\pi \in \mathcal{P}\left[s, t\right], \left|\pi\right| \rightarrow 0}\left\|\hat{\delta}\mathcal{A}_{s, t}-\sum_{\left[r, v\right] \in \pi}S_{v, t}A_{r, v}\right\|_{m, \gamma}=0, \quad \forall \left(s, t\right) \in \Delta_{2}.
	\end{equation*}
	Moreover, for every $\eta \in \left[0, z_{2}\right)$ we have 
	\begin{equation} \label{mssle}
		\left\|\hat{\delta}\mathcal{A}_{s, t}-A_{s, t}\right\|_{m, \gamma+\eta} \lesssim K_{1}\left|t-s\right|^{z_{1}-\left(\eta+z_{1}-z_{2}\right)^{+}}+K_{2}\left|t-s\right|^{z_{2}-\eta}, \quad \forall \left(s, t\right) \in \Delta_{2}.
	\end{equation}
\end{lemma}
\begin{proof}
	The existence and uniqueness of the process $\mathcal{A}$ can be obtained 
	analogous to the proof of \cite[Proposition 3.1 (i)]{LS22}. 
	We only show the estimate \eqref{mssle}. 
	For fixed $\left(s, t\right) \in \Delta_{2}$, 
	let $\pi_{n}$ be the dyadic partition of interval $\left[s, t\right]$, i.e. 
	\begin{equation*}
		\pi_{n}:=\left\{\left[s+k2^{-n}\left|t-s\right|, s+\left(k+1\right)2^{-n}\left|t-s\right|\right]: k=0, 1, \cdots, 2^{n}-1\right\}, \quad n=0, 1, \cdots. 
	\end{equation*}
	Define 
	\begin{equation*}
		A_{s, t}^{n}:=\sum_{\left[r, v\right] \in \pi_{n}}S_{v, t}A_{r, v}, \quad n=0, 1, \cdots.
	\end{equation*}
	Then $A_{s, t}^{0}=A_{s, t}$ and 
	\begin{align*}
		A_{s, t}^{n+1}-A_{s, t}^{n}&=\sum_{\left[r, v\right] \in \pi_{n}}\left(S_{\frac{r+v}{2}, t}A_{r, \frac{r+v}{2}}+S_{v, t}A_{\frac{r+v}{2}, v}-S_{v, t}A_{r, v}\right)\\
		&=-\sum_{\left[r, v\right] \in \pi_{n}}S_{v, t}\hat{\delta}A_{r, \frac{r+v}{2}, v}=-\sum_{\left[r, v\right] \in \pi_{n}}S_{\frac{r+v}{2}, t}\Lambda_{r, \frac{r+v}{2}, v}, \quad n=0, 1, \cdots.
	\end{align*}
	Taking $2\left(\eta+z_{1}-z_{2}\right)-1 < \kappa_{1} < 2z_{1}-1$ and $\eta-1 < \kappa_{2} < z_{2}-1$, 
	applying Proposition \ref{asp}, Minkowski's inequality and the BDG inequality, we have 
	\begin{align*}
		&\left\|A_{s, t}^{n+1}-A_{s, t}^{n}\right\|_{m, \gamma+\eta}=\left\|\sum_{\left[r, v\right] \in \pi_{n}}S_{\frac{r+v}{2}, t}\left(\Lambda_{r, \frac{r+v}{2}, v}-\mathbb{E}_{r}\Lambda_{r, \frac{r+v}{2}, v}\right)+\sum_{\left[r, v\right] \in \pi_{n}}S_{\frac{r+v}{2}, t}\mathbb{E}_{r}\Lambda_{r, \frac{r+v}{2}, v}\right\|_{m, \gamma+\eta}\\
		& \lesssim \left(\sum_{\left[r, v\right] \in \pi_{n}}\left\|S_{\frac{r+v}{2}, t}\left(\Lambda_{r, \frac{r+v}{2}, v}-\mathbb{E}_{r}\Lambda_{r, \frac{r+v}{2}, v}\right)\right\|_{m, \gamma+\eta}^{2}\right)^{\frac{1}{2}}+\sum_{\left[r, v\right] \in \pi_{n}}\left\|S_{\frac{r+v}{2}, t}\mathbb{E}_{r}\Lambda_{r, \frac{r+v}{2}, v}\right\|_{m, \gamma+\eta}\\
		& \lesssim \left(\sum_{\left[r, v\right] \in \pi_{n}}\left\|S_{\frac{r+v}{2}, t}\Lambda_{r, \frac{r+v}{2}, v}\right\|_{m, \gamma+\eta}^{2}\right)^{\frac{1}{2}}+\sum_{\left[r, v\right] \in \pi_{n}}\left\|S_{\frac{r+v}{2}, t}\mathbb{E}_{r}\Lambda_{r, \frac{r+v}{2}, v}\right\|_{m, \gamma+\eta}\\
		& \lesssim \left(\sum_{\left[r, v\right] \in \pi_{n}}\left|t-\frac{r+v}{2}\right|^{-2\left(\eta+z_{1}-z_{2}\right)^{+}}\left\|\Lambda_{r, \frac{r+v}{2}, v}\right\|_{m, \gamma+z_{2}-z_{1}}^{2}\right)^{\frac{1}{2}}+\sum_{\left[r, v\right] \in \pi_{n}}\left|t-\frac{r+v}{2}\right|^{-\eta}\left\|\mathbb{E}_{r}\Lambda_{r, \frac{r+v}{2}, v}\right\|_{m, \gamma}\\
		& \lesssim K_{1}\left(\sum_{\left[r, v\right] \in \pi_{n}}\left|t-\frac{r+v}{2}\right|^{\kappa_{1}-2\left(\eta+z_{1}-z_{2}\right)^{+}}\left|v-r\right|\left|\frac{v-r}{2}\right|^{2z_{1}-1-\kappa_{1}}\right)^{\frac{1}{2}}\\
		& \quad+K_{2}\sum_{\left[r, v\right] \in \pi_{n}}\left|t-\frac{r+v}{2}\right|^{\kappa_{2}-\eta}\left|v-r\right|\left|\frac{v-r}{2}\right|^{z_{2}-1-\kappa_{2}}\\
		& \lesssim K_{1}\left(\frac{\left|t-s\right|}{2^{n+1}}\right)^{z_{1}-\frac{1}{2}\left(1+\kappa_{1}\right)}\left(\int_{s}^{t}\left|t-r\right|^{\kappa_{1}-2\left(\eta+z_{1}-z_{2}\right)^{+}}dr\right)^{\frac{1}{2}}+K_{2}\left(\frac{\left|t-s\right|}{2^{n+1}}\right)^{z_{2}-1-\kappa_{2}}\int_{s}^{t}\left|t-r\right|^{\kappa_{2}-\eta}dr\\
		& \lesssim K_{1}2^{-\left(n+1\right)\left(z_{1}-\frac{1+\kappa_{1}}{2}\right)}\left|t-s\right|^{z_{1}-\left(\eta+z_{1}-z_{2}\right)^{+}}+K_{2}2^{-\left(n+1\right)\left(z_{2}-1-\kappa_{2}\right)}\left|t-s\right|^{z_{2}-\eta}, \quad n=0, 1, \cdots.
	\end{align*}
	Hence, 
	\begin{equation*}
		\left\|\hat{\delta}\mathcal{A}_{s, t}-A_{s, t}\right\|_{m, \gamma+\eta} \lesssim \sum_{n=0}^{\infty}\left\|A_{s, t}^{n+1}-A_{s, t}^{n}\right\|_{m, \gamma+\eta} \lesssim K_{1}\left|t-s\right|^{z_{1}-\left(\eta+z_{1}-z_{2}\right)^{+}}+K_{2}\left|t-s\right|^{z_{2}-\eta}.
	\end{equation*}
\end{proof}
Now we give the following result on rough stochastic convolutions. 
\begin{theorem} \label{rsci}
	Let $\mathbf{X}=\left(X, \mathbb{X}\right) \in \mathscr{C}^{\alpha}$ and 
	$\left(Y, Y^{\prime}\right) \in \mathbf{D}_{X}^{\beta, \beta^{\prime}}L_{m}\mathcal{H}_{\gamma}^{e}$ 
	with $\alpha+\beta+\beta^{\prime} > 1$. 
	Then there exists unique 
	\begin{equation*}
		Z:=\int_{0}^{\cdot}S_{r, \cdot}Y_{r}d\mathbf{X}_{r} \in \hat{C}^{\alpha}L_{m}\mathcal{H}_{\gamma-\beta-\beta^{\prime}} 
	\end{equation*}
	with $Z_{0}=0$ such that for every $\left(s, t\right) \in \Delta_{2}$, 
	\begin{equation} \label{rscid}
		\hat{\delta}Z_{s, t}=\int_{s}^{t}S_{r, t}Y_{r}d\mathbf{X}_{r}=\lim_{\pi \in \mathcal{P}\left[s, t\right], \left|\pi\right| \rightarrow 0}\sum_{\left[r, v\right] \in \pi}S_{r, t}\left(Y_{r}\delta X_{r, v}+Y_{r}^{\prime}\mathbb{X}_{r, v}\right)
	\end{equation}
	holds in $L^{m}\left(\Omega, \mathcal{H}_{\gamma-\beta-\beta^{\prime}}\right)$, 
	and for every $\eta \in \left[0, \alpha+\beta+\beta^{\prime}\right)$ and 
	$\left(s, t\right) \in \Delta_{2}$ we have 
	\begin{equation} \label{rscie1}
		\left\|\hat{\delta}Z_{s, t}-S_{s, t}\left(Y_{s}\delta X_{s, t}+Y_{s}^{\prime}\mathbb{X}_{s, t}\right)\right\|_{m, \gamma+\eta-\beta-\beta^{\prime}} \lesssim \left\|Y, Y^{\prime}\right\|_{\mathbf{D}_{X}^{\beta, \beta^{\prime}}L_{m}\mathcal{H}_{\gamma}}\left(1+\left|\mathbf{X}\right|_{\alpha}^{2}\right)\left|t-s\right|^{\alpha+\beta+\beta^{\prime}-\eta}.
	\end{equation}
	We call $\int_{0}^{\cdot}S_{r, \cdot}Y_{r}d\mathbf{X}_{r}$ the rough stochastic convolution of $\left(Y, Y^{\prime}\right)$ 
	against $\mathbf{X}$. 
	As a consequence, we have $Z \in \hat{C}^{\alpha-\theta}L_{m}\mathcal{H}_{\gamma+\theta}$ for every $\theta \in \left[0, \alpha\right)$ and 
	\begin{equation} \label{rscie2}
		\left\|\hat{\delta}Z\right\|_{\alpha-\theta, m, \gamma+\theta} \lesssim \left\|Y, Y^{\prime}\right\|_{\mathbf{D}_{X}^{\beta, \beta^{\prime}}L_{m}\mathcal{H}_{\gamma}}\left(1+\left|\mathbf{X}\right|_{\alpha}^{2}\right).
	\end{equation}
	Moreover, if additionally $\theta \leq \beta^{\prime}$, then the rough stochastic convolution map 
	\begin{equation*}
		\mathbf{D}_{X}^{\beta, \beta^{\prime}}L_{m}\mathcal{H}_{\gamma}^{e} \rightarrow \mathcal{D}_{X}^{\beta, \beta^{\prime}}L_{m}\mathcal{H}_{\gamma+\theta}: \quad \left(Y, Y^{\prime}\right) \mapsto \left(Z, Y\right)
	\end{equation*}
	is a bounded linear map and we have 
	\begin{equation} \label{rscie3}
		\left\|Z, Y\right\|_{\mathbf{D}_{X}^{\beta, \beta^{\prime}}L_{m}\mathcal{H}_{\gamma+\theta}} \lesssim \left\|Y_{0}\right\|_{m, \gamma}\left(1+\left|\mathbf{X}\right|_{\alpha}\right)+T^{\left(\alpha-\beta\right)\wedge\left(\beta-\beta^{\prime}\right)}\left\|Y, Y^{\prime}\right\|_{\mathbf{D}_{X}^{\beta, \beta^{\prime}}L_{m}\mathcal{H}_{\gamma}}\left(1+\left|\mathbf{X}\right|_{\alpha}^{2}\right).
	\end{equation}
\end{theorem}
\begin{proof}
	Define 
	\begin{equation} \label{ad}
		A_{s, t}:=S_{s, t}\left(Y_{s}\delta X_{s, t}+Y_{s}^{\prime}\mathbb{X}_{s, t}\right), \quad \forall \left(s, t\right) \in \Delta_{2}. 
	\end{equation}
	Clearly, $A$ is measurable and adapted, and 
	$A \in C\left(\Delta_{2}, L^{m}\left(\Omega, \mathcal{H}_{\gamma-\beta-\beta^{\prime}}\right)\right)$. 
	By Proposition \ref{asp}, 
	\begin{align*}
		\left\|A_{s, t}\right\|_{m, \gamma-\beta-\beta^{\prime}} &\lesssim \left\|Y_{s}\right\|_{m, \gamma-\beta-\beta^{\prime}}\left|\delta X_{s, t}\right|+\left\|Y_{s}^{\prime}\right\|_{m, \gamma-\beta-\beta^{\prime}}\left|\mathbb{X}_{s, t}\right|\\
		&\lesssim \left\|Y, Y^{\prime}\right\|_{\mathbf{D}_{X}^{\beta, \beta^{\prime}}L_{m}\mathcal{H}_{\gamma}}\left|\mathbf{X}\right|_{\alpha}\left|t-s\right|^{\alpha}, \quad \forall \left(s, t\right) \in \Delta_{2},
	\end{align*}
	which gives $A \in C_{2}^{\alpha}L_{m}\mathcal{H}_{\gamma-\beta-\beta^{\prime}}$. 
	In view of Chen's relation \eqref{Chen}, we have 
	\begin{equation*}
		\hat{\delta} A_{s, r, t}=-S_{r, t}\left(\hat{R}_{s, r}^{Y}\delta X_{r, t}+\hat{\delta}Y_{s, r}^{\prime}\mathbb{X}_{r, t}\right), \quad \forall \left(s, r, t\right) \in \Delta_{3}.
	\end{equation*}
	Define 
	\begin{equation*}
		\Lambda_{s, r, t}:=-\hat{R}_{s, r}^{Y}\delta X_{r, t}-\hat{\delta}Y_{s, r}^{\prime}\mathbb{X}_{r, t}, \quad \forall \left(s, r, t\right) \in \Delta_{3}.
	\end{equation*}
	By Propositions \ref{asp}, \ref{ne} and \ref{de}, we have 
	\begin{align*}
		\left\|\Lambda_{s, r, t}\right\|_{m, \gamma-\beta} &\lesssim \left(\left\|\hat{\delta}Y_{s, r}\right\|_{m, \gamma-\beta}+\left\|Y_{s}^{\prime}\right\|_{m, \gamma-\beta}\left|\delta X_{s, r}\right|\right)\left|\delta X_{r, t}\right|+\left\|\hat{\delta} Y_{s, r}^{\prime}\right\|_{m, \gamma-\beta}\left|\mathbb{X}_{r, t}\right|\\
		&\lesssim \left\|Y, Y^{\prime}\right\|_{\mathbf{D}_{X}^{\beta, \beta^{\prime}}L_{m}\mathcal{H}_{\gamma}}\left(1+\left|\mathbf{X}\right|_{\alpha}^{2}\right)\left|t-s\right|^{\frac{1}{2}}\left|t-r\right|^{\alpha+\beta-\frac{1}{2}},
	\end{align*}
	\begin{align*}
		\left\|\mathbb{E}_{s}\Lambda_{s, r, t}\right\|_{m, \gamma-\beta-\beta^{\prime}} &\lesssim \left(\left\|\mathbb{E}_{s}R_{s, r}^{Y}\right\|_{m, \gamma-\beta-\beta^{\prime}}+\left\|R_{s, r}^{Y}-\hat{R}_{s, r}^{Y}\right\|_{m, \gamma-\beta-\beta^{\prime}}\right)\left|\delta X_{r, t}\right|+\left\|\hat{\delta} Y_{s, r}^{\prime}\right\|_{m, \gamma-\beta-\beta^{\prime}}\left|\mathbb{X}_{r, t}\right|\\
		&\lesssim \left\|Y, Y^{\prime}\right\|_{\mathbf{D}_{X}^{\beta, \beta^{\prime}}L_{m}\mathcal{H}_{\gamma}}\left(1+\left|\mathbf{X}\right|_{\alpha}^{2}\right)\left|t-s\right|\left|t-r\right|^{\alpha+\beta+\beta^{\prime}-1}, \quad \forall \left(s, r, t\right) \in \Delta_{3}.
	\end{align*}
	Then by Lemma \ref{mssl} with $z_{1}=\alpha+\beta$ and $z_{2}=\alpha+\beta+\beta^{\prime}$, there exists unique $Z \in \hat{C}^{\alpha}L_{m}\mathcal{H}_{\gamma-\beta-\beta^{\prime}}$ 
	with $Z_{0}=0$ such that \eqref{rscid} holds in $L^{m}\left(\Omega, \mathcal{H}_{\gamma-\beta-\beta^{\prime}}\right)$,  
	and for every $\eta \in \left[0, \alpha+\beta+\beta^{\prime}\right)$ the estimate \eqref{rscie1} holds. 
	By Proposition \ref{asp}, for every $\theta \in \left[0, \alpha\right)$ we have 
	\begin{align*}
		\left\|A_{s, t}\right\|_{m, \gamma+\theta} &\lesssim \left|t-s\right|^{-\theta}\left\|Y_{s}\right\|_{m, \gamma}\left|\delta X_{s, t}\right|+\left|t-s\right|^{-\left(\beta+\theta\right)}\left\|Y_{s}^{\prime}\right\|_{m, \gamma-\beta}\left|\mathbb{X}_{s, t}\right|\\
		&\lesssim \left\|Y, Y^{\prime}\right\|_{\mathbf{D}_{X}^{\beta, \beta^{\prime}}L_{m}\mathcal{H}_{\gamma}}\left|\mathbf{X}\right|_{\alpha}\left|t-s\right|^{\alpha-\theta}, \quad \forall \left(s, t\right) \in \Delta_{2}.
	\end{align*}
	In view of \eqref{rscie1} with $\eta=\beta+\beta^{\prime}+\theta$, we have $Z \in \hat{C}^{\alpha-\theta}L_{m}\mathcal{H}_{\gamma+\theta}$ and the estimate \eqref{rscie2} holds. 
	Moreover, for $\theta \leq \beta^{\prime}$ we have 
	\begin{equation*}
		\left\|\hat{\delta}Z\right\|_{\beta, m, \gamma+\theta-\beta} \lesssim T^{\alpha-\beta}\left\|\hat{\delta}Z\right\|_{\alpha, m, \gamma} \lesssim T^{\alpha-\beta}\left\|Y, Y^{\prime}\right\|_{\mathbf{D}_{X}^{\beta, \beta^{\prime}}L_{m}\mathcal{H}_{\gamma}}\left(1+\left|\mathbf{X}\right|_{\alpha}^{2}\right).
	\end{equation*}
	Since $Z_{0}=0$ we have 
	\begin{equation*}
		\left\|Z\right\|_{0, m, \gamma+\theta} \lesssim T^{\alpha-\theta}\left\|\hat{\delta}Z\right\|_{\alpha-\theta, m, \gamma+\theta} \lesssim T^{\alpha-\theta}\left\|Y, Y^{\prime}\right\|_{\mathbf{D}_{X}^{\beta, \beta^{\prime}}L_{m}\mathcal{H}_{\gamma}}\left(1+\left|\mathbf{X}\right|_{\alpha}^{2}\right).
	\end{equation*}
	By Proposition \ref{ne}, we have $Z \in E^{\beta}L_{m}\mathcal{H}_{\gamma+\theta}$ and 
	\begin{equation} \label{rscie4}
		\left\|Z\right\|_{E^{\beta}L_{m}\mathcal{H}_{\gamma+\theta}} \lesssim T^{\alpha-\beta}\left\|Y, Y^{\prime}\right\|_{\mathbf{D}_{X}^{\beta, \beta^{\prime}}L_{m}\mathcal{H}_{\gamma}}\left(1+\left|\mathbf{X}\right|_{\alpha}^{2}\right).
	\end{equation}
	On the other hand, by proposition \ref{in} we have $Y \in E^{\beta}L_{m}\mathcal{H}_{\gamma}^{e} \subset E^{\beta-\theta}L_{m}\mathcal{H}_{\gamma}^{e}$ and 
	\begin{equation*}
		\left\|Y\right\|_{0, m, \gamma+\theta-\beta} \lesssim \left\|Y_{0}\right\|_{m, \gamma+\theta-\beta}+T^{\beta-\theta}\left\|\delta Y\right\|_{\beta-\theta, m, \gamma+\theta-\beta} \lesssim \left\|Y_{0}\right\|_{m, \gamma}+T^{\beta-\theta}\left\|Y, Y^{\prime}\right\|_{\mathbf{D}_{X}^{\beta, \beta^{\prime}}L_{m}\mathcal{H}_{\gamma}}. 
	\end{equation*}
	Combined with  
	\begin{equation*}
		\left\|\delta Y\right\|_{\beta^{\prime}, m, \gamma+\theta-\beta-\beta^{\prime}} \lesssim T^{\beta-\beta^{\prime}}\left\|\delta Y\right\|_{\beta, m, \gamma-\beta} \lesssim T^{\beta-\beta^{\prime}}\left\|Y, Y^{\prime}\right\|_{\mathbf{D}_{X}^{\beta, \beta^{\prime}}L_{m}\mathcal{H}_{\gamma}}, 
	\end{equation*}
	we have $Y \in E^{\beta^{\prime}}L_{m}\mathcal{H}_{\gamma+\theta-\beta}$ and 
	\begin{equation} \label{rscie5}
		\left\|Y\right\|_{E^{\beta^{\prime}}L_{m}\mathcal{H}_{\gamma+\theta-\beta}} \lesssim \left\|Y_{0}\right\|_{m, \gamma}+T^{\beta-\beta^{\prime}}\left\|Y, Y^{\prime}\right\|_{\mathbf{D}_{X}^{\beta, \beta^{\prime}}L_{m}\mathcal{H}_{\gamma}}. 
	\end{equation}
	At last, by Proposition \ref{asp}, 
	\begin{equation*}
		\left\|S_{s, t}Y_{s}^{\prime}\mathbb{X}_{s, t}\right\|_{m, \gamma+\theta-\beta-\beta^{\prime}} \lesssim \left\|Y_{s}^{\prime}\right\|_{m, \gamma-\beta}\left|\mathbb{X}_{s, t}\right| \lesssim \left\|Y, Y^{\prime}\right\|_{\mathbf{D}_{X}^{\beta, \beta^{\prime}}L_{m}\mathcal{H}_{\gamma}}\left|\mathbf{X}\right|_{\alpha}\left|t-s\right|^{2\alpha}, \quad \forall \left(s, t\right) \in \Delta_{2}.
	\end{equation*}
	In view of \eqref{rscie1} with $\eta=\theta$, we have 
	\begin{align*}
		\left\|\hat{R}_{s, t}^{Z}\right\|_{m, \gamma+\theta-\beta-\beta^{\prime}} &\lesssim \left\|\hat{\delta}Z_{s, t}-A_{s, t}\right\|_{m, \gamma+\theta-\beta-\beta^{\prime}}+\left\|S_{s, t}Y_{s}^{\prime}\mathbb{X}_{s, t}\right\|_{m, \gamma+\theta-\beta-\beta^{\prime}}\\
		&\lesssim T^{\alpha-\beta^{\prime}}\left\|Y, Y^{\prime}\right\|_{\mathbf{D}_{X}^{\beta, \beta^{\prime}}L_{m}\mathcal{H}_{\gamma}}\left(1+\left|\mathbf{X}\right|_{\alpha}^{2}\right)\left|t-s\right|^{\beta+\beta^{\prime}}, \quad \forall \left(s, t\right) \in \Delta_{2},
	\end{align*}
	which gives $\hat{R}^{Z} \in C_{2}^{\beta+\beta^{\prime}}L_{m}\mathcal{H}_{\gamma+\theta-\beta-\beta^{\prime}}$ and 
	\begin{equation*}
		\left\|\mathbb{E}_{\cdot}\hat{R}^{Z}\right\|_{\beta+\beta^{\prime}, m, \gamma+\theta-\beta-\beta^{\prime}} \lesssim \left\|\hat{R}^{Z}\right\|_{\beta+\beta^{\prime}, m, \gamma+\theta-\beta-\beta^{\prime}} \lesssim T^{\alpha-\beta^{\prime}}\left\|Y, Y^{\prime}\right\|_{\mathbf{D}_{X}^{\beta, \beta^{\prime}}L_{m}\mathcal{H}_{\gamma}}\left(1+\left|\mathbf{X}\right|_{\alpha}^{2}\right).
	\end{equation*}
	Combining \eqref{dei}, \eqref{rscie4}, \eqref{rscie5} and the last inequality, 
	we have $R^{Z} \in C_{2}^{\beta+\beta^{\prime}}L_{m}\mathcal{H}_{\gamma+\theta-\beta-\beta^{\prime}}$ and 
	\begin{align*}
		\left\|\mathbb{E}_{\cdot}R^{Z}\right\|_{\beta+\beta^{\prime}, m, \gamma+\theta-\beta-\beta^{\prime}} &\lesssim \left\|\mathbb{E}_{\cdot}\hat{R}^{Z}\right\|_{\beta+\beta^{\prime}, m, \gamma+\theta-\beta-\beta^{\prime}}+\left\|Z\right\|_{E^{\beta}L_{m}\mathcal{H}_{\gamma+\theta}}+\left\|Y\right\|_{E^{\beta^{\prime}}L_{m}\mathcal{H}_{\gamma+\theta-\beta}}\left|\mathbf{X}\right|_{\alpha}\\
		&\lesssim \left\|Y_{0}\right\|_{m, \gamma}\left|\mathbf{X}\right|_{\alpha}+T^{\left(\alpha-\beta\right)\wedge\left(\beta-\beta^{\prime}\right)}\left\|Y, Y^{\prime}\right\|_{\mathbf{D}_{X}^{\beta, \beta^{\prime}}L_{m}\mathcal{H}_{\gamma}}\left(1+\left|\mathbf{X}\right|_{\alpha}^{2}\right).
	\end{align*}
	Therefore, $\left(Z, Y\right) \in \mathcal{D}_{X}^{\beta, \beta^{\prime}}L_{m}\mathcal{H}_{\gamma+\theta}$ 
	and the estimate \eqref{rscie3} holds. 
\end{proof}
Then we state the stability of rough stochastic convolution. 
\begin{theorem} \label{rscs}
	Let $\mathbf{X}=\left(X, \mathbb{X}\right), \bar{\mathbf{X}}=\left(\bar{X}, \bar{\mathbb{X}}\right) \in \mathscr{C}^{\alpha}$, 
	$\left(Y, Y^{\prime}\right) \in \mathbf{D}_{X}^{\beta, \beta^{\prime}}L_{m}\mathcal{H}_{\gamma}^{e}$, 
	$\left(\bar{Y}, \bar{Y}^{\prime}\right) \in \mathbf{D}_{\bar{X}}^{\beta, \beta^{\prime}}L_{m}\mathcal{H}_{\gamma}^{e}$ 
	with $\alpha+\beta+\beta^{\prime} > 1$. 
	Assume $\left|\mathbf{X}\right|_{\alpha}, \left|\bar{\mathbf{X}}\right|_{\alpha}, \left\|Y, Y^{\prime}\right\|_{\mathbf{D}_{X}^{\beta, \beta^{\prime}}L_{m}\mathcal{H}_{\gamma}}, \left\|\bar{Y}, \bar{Y}^{\prime}\right\|_{\mathbf{D}_{\bar{X}}^{\beta, \beta^{\prime}}L_{m}\mathcal{H}_{\gamma}} \leq R$ 
	for some $R \geq 0$. 
	Define $Z:=\int_{0}^{\cdot}S_{r, \cdot}Y_{r}d\mathbf{X}_{r}$ and 
	$\bar{Z}:=\int_{0}^{\cdot}S_{r, \cdot}\bar{Y}_{r}d\bar{\mathbf{X}}_{r}$. 
	Then for every $\theta \in \left[0, \beta^{\prime}\right]$ we have 
	\begin{equation} \label{rscse1}
		\left\|Z, Y; \bar{Z}, \bar{Y}\right\|_{\mathbf{D}_{X, \bar{X}}^{\beta, \beta^{\prime}}L_{m}\mathcal{H}_{\gamma+\theta}} \lesssim \rho_{\alpha}\left(\mathbf{X}, \bar{\mathbf{X}}\right)+\left\|\Delta Y_{0}\right\|_{m, \gamma}+T^{\left(\alpha-\beta\right)\wedge\left(\beta-\beta^{\prime}\right)}\left\|Y, Y^{\prime}; \bar{Y}, \bar{Y}^{\prime}\right\|_{\mathbf{D}_{X, \bar{X}}^{\beta, \beta^{\prime}}L_{m}\mathcal{H}_{\gamma}},
	\end{equation}
	for a hidden prefactor depending on $R$. 
\end{theorem}
\begin{proof}
	Like defining $A$ by \eqref{ad}, define 
	\begin{equation*}
		\bar{A}_{s, t}:=S_{s, t}\left(\bar{Y}_{s}\delta \bar{X}_{s, t}+\bar{Y}_{s}^{\prime}\bar{\mathbb{X}}_{s, t}\right), \quad \forall \left(s, t\right) \in \Delta_{2}.
	\end{equation*}
	In view of Chen's relation \eqref{Chen}, we have 
	\begin{equation*}
		\hat{\delta}\Delta A_{s, r, t}=-S_{r, t}\left(\hat{R}_{s, r}^{Y}\delta \Delta X_{r, t}+\hat{\delta}Y_{s, r}^{\prime}\Delta\mathbb{X}_{r, t}+\Delta\hat{R}_{s, r}^{Y}\delta \bar{X}_{r, t}+\hat{\delta}\Delta Y_{s, r}^{\prime}\bar{\mathbb{X}}_{r, t}\right), \quad \forall \left(s, r, t\right) \in \Delta_{3}.
	\end{equation*}
	By Propositions \ref{asp} and \ref{ne}, in view of \eqref{des} we have 
	\begin{align*}
		&\left\|\hat{R}_{s, r}^{Y}\delta \Delta X_{r, t}+\hat{\delta}Y_{s, r}^{\prime}\Delta\mathbb{X}_{r, t}+\Delta\hat{R}_{s, r}^{Y}\delta \bar{X}_{r, t}+\hat{\delta}\Delta Y_{s, r}^{\prime}\bar{\mathbb{X}}_{r, t}\right\|_{m, \gamma-\beta}\\
		&\lesssim \left(\left\|\hat{\delta}Y_{s, r}\right\|_{m, \gamma-\beta}+\left\|Y_{s}^{\prime}\right\|_{m, \gamma-\beta}\left|\delta X_{s, r}\right|\right)\left|\delta \Delta X_{r, t}\right|+\left\|\hat{\delta} Y_{s, r}^{\prime}\right\|_{m, \gamma-\beta}\left|\Delta \mathbb{X}_{r, t}\right|\\
		&\quad+\left(\left\|\hat{\delta}\Delta Y_{s, r}\right\|_{m, \gamma-\beta}+\left\|Y_{s}^{\prime}\right\|_{m, \gamma-\beta}\left|\delta \Delta X_{s, r}\right|+\left\|\Delta Y_{s}^{\prime}\right\|_{m, \gamma-\beta}\left|\delta \bar{X}_{s, r}\right|\right)\left|\delta \bar{X}_{r, t}\right|+\left\|\hat{\delta} \Delta Y_{s, r}^{\prime}\right\|_{m, \gamma-\beta}\left|\bar{\mathbb{X}}_{r, t}\right|\\
		&\lesssim \left(\rho_{\alpha}\left(\mathbf{X}, \bar{\mathbf{X}}\right)+\left\|Y, Y^{\prime}; \bar{Y}, \bar{Y}^{\prime}\right\|_{\mathbf{D}_{X, \bar{X}}^{\beta, \beta^{\prime}}L_{m}\mathcal{H}_{\gamma}}\right)\left|t-s\right|^{\frac{1}{2}}\left|t-r\right|^{\alpha+\beta-\frac{1}{2}},
	\end{align*}
	\begin{align*}
		&\left\|\mathbb{E}_{s}\left(\hat{R}_{s, r}^{Y}\delta \Delta X_{r, t}+\hat{\delta}Y_{s, r}^{\prime}\Delta\mathbb{X}_{r, t}+\Delta\hat{R}_{s, r}^{Y}\delta \bar{X}_{r, t}+\hat{\delta}\Delta Y_{s, r}^{\prime}\bar{\mathbb{X}}_{r, t}\right)\right\|_{m, \gamma-\beta-\beta^{\prime}}\\
		&\lesssim \left(\left\|\mathbb{E}_{s}R_{s, r}^{Y}\right\|_{m, \gamma-\beta-\beta^{\prime}}+\left\|R_{s, r}^{Y}-\hat{R}_{s, r}^{Y}\right\|_{m, \gamma-\beta-\beta^{\prime}}\right)\left|\delta \Delta X_{r, t}\right|+\left\|\hat{\delta} Y_{s, r}^{\prime}\right\|_{m, \gamma-\beta-\beta^{\prime}}\left|\Delta \mathbb{X}_{r, t}\right|\\
		&\quad+\left(\left\|\mathbb{E}_{s}\Delta R_{s, r}^{Y}\right\|_{m, \gamma-\beta-\beta^{\prime}}+\left\|\Delta R_{s, r}^{Y}-\Delta\hat{R}_{s, r}^{Y}\right\|_{m, \gamma-\beta-\beta^{\prime}}\right)\left|\delta \bar{X}_{r, t}\right|+\left\|\hat{\delta} \Delta Y_{s, r}^{\prime}\right\|_{m, \gamma-\beta-\beta^{\prime}}\left|\bar{\mathbb{X}}_{r, t}\right|\\
		&\lesssim \left(\rho_{\alpha}\left(\mathbf{X}, \bar{\mathbf{X}}\right)+\left\|Y, Y^{\prime}; \bar{Y}, \bar{Y}^{\prime}\right\|_{\mathbf{D}_{X, \bar{X}}^{\beta, \beta^{\prime}}L_{m}\mathcal{H}_{\gamma}}\right)\left|t-s\right|\left|t-r\right|^{\alpha+\beta+\beta^{\prime}-1}, \quad \forall \left(s, r, t\right) \in \Delta_{3}.
	\end{align*}
	Then by Lemma \ref{mssl}, for every $\eta \in \left[0, \alpha+\beta+\beta^{\prime}\right)$ 
	and $\left(s, t\right) \in \Delta_{2}$ we have 
	\begin{equation} \label{rscse2}
		\left\|\hat{\delta}\Delta Z_{s, t}-\Delta A_{s, t}\right\|_{m, \gamma+\eta-\beta-\beta^{\prime}} \lesssim \left(\rho_{\alpha}\left(\mathbf{X}, \bar{\mathbf{X}}\right)+\left\|Y, Y^{\prime}; \bar{Y}, \bar{Y}^{\prime}\right\|_{\mathbf{D}_{X, \bar{X}}^{\beta, \beta^{\prime}}L_{m}\mathcal{H}_{\gamma}}\right)\left|t-s\right|^{\alpha+\beta+\beta^{\prime}-\eta}.
	\end{equation}
	By Proposition \ref{asp}, for every $\theta \in \left[0, \beta^{\prime}\right]$ we have 
	\begin{align*}
		\left\|\Delta A_{s, t}\right\|_{m, \gamma+\theta} &\lesssim \left|t-s\right|^{-\theta}\left\|Y_{s}\right\|_{m, \gamma}\left|\delta \Delta X_{s, t}\right|+\left|t-s\right|^{-\left(\beta+\theta\right)}\left\|Y_{s}^{\prime}\right\|_{m, \gamma-\beta}\left|\Delta \mathbb{X}_{s, t}\right|\\
		&\quad+\left|t-s\right|^{-\theta}\left\|\Delta Y_{s}\right\|_{m, \gamma}\left|\delta \bar{X}_{s, t}\right|+\left|t-s\right|^{-\left(\beta+\theta\right)}\left\|\Delta Y_{s}^{\prime}\right\|_{m, \gamma-\beta}\left|\Delta \bar{\mathbb{X}}_{s, t}\right|\\
		&\lesssim \left(\rho_{\alpha}\left(\mathbf{X}, \bar{\mathbf{X}}\right)+\left\|Y, Y^{\prime}; \bar{Y}, \bar{Y}^{\prime}\right\|_{\mathbf{D}_{X, \bar{X}}^{\beta, \beta^{\prime}}L_{m}\mathcal{H}_{\gamma}}\right)\left|t-s\right|^{\alpha-\theta}, \quad \forall \left(s, t\right) \in \Delta_{2}.
	\end{align*}
	Hence, 
	\begin{equation*}
		\left\|\hat{\delta}\Delta Z\right\|_{\alpha-\theta, m, \gamma+\theta} \lesssim \rho_{\alpha}\left(\mathbf{X}, \bar{\mathbf{X}}\right)+\left\|Y, Y^{\prime}; \bar{Y}, \bar{Y}^{\prime}\right\|_{\mathbf{D}_{X, \bar{X}}^{\beta, \beta^{\prime}}L_{m}\mathcal{H}_{\gamma}}.
	\end{equation*}
	Analogous to the proof of estimates \eqref{rscie4} and \eqref{rscie5}, we have 
	\begin{equation} \label{rscse3}
		\left\|\Delta Z\right\|_{E^{\beta}L_{m}\mathcal{H}_{\gamma+\theta}} \lesssim \rho_{\alpha}\left(\mathbf{X}, \bar{\mathbf{X}}\right)+T^{\alpha-\beta}\left\|Y, Y^{\prime}; \bar{Y}, \bar{Y}^{\prime}\right\|_{\mathbf{D}_{X, \bar{X}}^{\beta, \beta^{\prime}}L_{m}\mathcal{H}_{\gamma}},
	\end{equation}
	\begin{equation} \label{rscse4}
		\left\|\Delta Y\right\|_{E^{\beta^{\prime}}L_{m}\mathcal{H}_{\gamma+\theta-\beta}} \lesssim \rho_{\alpha}\left(\mathbf{X}, \bar{\mathbf{X}}\right)+\left\|\Delta Y_{0}\right\|_{m, \gamma}+T^{\beta-\beta^{\prime}}\left\|Y, Y^{\prime}; \bar{Y}, \bar{Y}^{\prime}\right\|_{\mathbf{D}_{X, \bar{X}}^{\beta, \beta^{\prime}}L_{m}\mathcal{H}_{\gamma}}.
	\end{equation}
	At last, by Proposition \ref{asp}, for every $\left(s, t\right) \in \Delta_{2}$ we have 
	\begin{align*}
		\left\|S_{s, t}\left(Y_{s}^{\prime}\Delta\mathbb{X}_{s, t}+\Delta Y_{s}^{\prime}\bar{\mathbb{X}}_{s, t}\right)\right\|_{m, \gamma+\theta-\beta-\beta^{\prime}} &\lesssim \left\|Y_{s}^{\prime}\right\|_{m, \gamma-\beta}\left|\Delta\mathbb{X}_{s, t}\right|+\left\|\Delta Y_{s}^{\prime}\right\|_{m, \gamma-\beta}\left|\bar{\mathbb{X}}_{s, t}\right|\\
		&\lesssim \left(\rho_{\alpha}\left(\mathbf{X}, \bar{\mathbf{X}}\right)+\left\|Y, Y^{\prime}; \bar{Y}, \bar{Y}^{\prime}\right\|_{\mathbf{D}_{X, \bar{X}}^{\beta, \beta^{\prime}}L_{m}\mathcal{H}_{\gamma}}\right)\left|t-s\right|^{2\alpha}.
	\end{align*}
	In view of \eqref{rscse2}, we have 
	\begin{align*}
		\left\|\Delta \hat{R}_{s, t}^{Z}\right\|_{m, \gamma+\theta-\beta-\beta^{\prime}} &\lesssim \left\|\hat{\delta}\Delta Z_{s, t}-\Delta A_{s, t}\right\|_{m, \gamma+\theta-\beta-\beta^{\prime}}+\left\|S_{s, t}\left(Y_{s}^{\prime}\Delta\mathbb{X}_{s, t}+\Delta Y_{s}^{\prime}\bar{\mathbb{X}}_{s, t}\right)\right\|_{m, \gamma+\theta-\beta-\beta^{\prime}}\\
		&\lesssim \left(\rho_{\alpha}\left(\mathbf{X}, \bar{\mathbf{X}}\right)+T^{\alpha-\beta^{\prime}}\left\|Y, Y^{\prime}; \bar{Y}, \bar{Y}^{\prime}\right\|_{\mathbf{D}_{X, \bar{X}}^{\beta, \beta^{\prime}}L_{m}\mathcal{H}_{\gamma}}\right)\left|t-s\right|^{\beta+\beta^{\prime}}, \quad \forall \left(s, t\right) \in \Delta_{2},
	\end{align*}
	which gives 
	\begin{equation*}
		\left\|\mathbb{E}_{\cdot}\Delta\hat{R}^{Y}\right\|_{\beta+\beta^{\prime}, m, \gamma+\theta-\beta-\beta^{\prime}} \lesssim \rho_{\alpha}\left(\mathbf{X}, \bar{\mathbf{X}}\right)+T^{\alpha-\beta^{\prime}}\left\|Y, Y^{\prime}; \bar{Y}, \bar{Y}^{\prime}\right\|_{\mathbf{D}_{X, \bar{X}}^{\beta, \beta^{\prime}}L_{m}\mathcal{H}_{\gamma}}.
	\end{equation*}
	Combining \eqref{des}, \eqref{rscse3}, \eqref{rscse4} and the last inequality, 
	we get the estimate \eqref{rscse1}. 
\end{proof}
\subsection{Stochastic controlled operator families}
Next, we introduce stochastic controlled operator families according to a monotone family. 
\begin{definition}
	Given $X \in C^{\alpha}\left(\left[0, T\right], \mathbb{R}^{e}\right)$, 
	we call $\left(G, G^{\prime}\right)$ an (essentially bounded)
	$\left(\beta, \beta^{\prime}\right)$-Hölder 
	stochastic controlled operator family (controlled by $X$) 
	from $\mathcal{H}_{\gamma_{1}}^{e_{1}}$ to $\mathcal{H}_{\gamma_{2}}^{e_{2}}$, 
	denoted by $\left(G, G^{\prime}\right) \in \mathbf{D}_{X}^{\beta, \beta^{\prime}}L_{\infty}\mathcal{L}_{\gamma_{1}, \gamma_{2}}\left(\mathcal{H}^{e_{1}}, \mathcal{H}^{e_{2}}\right)$, 
	if $G \in E^{\beta}L_{\infty}\mathcal{L}_{\gamma_{1}, \gamma_{2}}\left(\mathcal{H}^{e_{1}}, \mathcal{H}^{e_{2}}\right) \cap E^{\beta^{\prime}}L_{\infty}\mathcal{L}_{\gamma_{1}-\beta, \gamma_{2}-\beta}\left(\mathcal{H}^{e_{1}}, \mathcal{H}^{e_{2}}\right)$, 
	$G^{\prime} \in E^{\beta^{\prime}}L_{\infty}\mathcal{L}_{\gamma_{1}, \gamma_{2}-\beta}\left(\mathcal{H}^{e_{1}}, \mathcal{H}^{e_{2} \times e}\right)$ and 
	$\mathbb{E}_{\cdot}R^{G} \in C_{2}^{\beta+\beta^{\prime}}L_{\infty}\mathcal{L}\left(\mathcal{H}_{\gamma_{1}}^{e_{1}}, \mathcal{H}_{\gamma_{2}-\beta-\beta^{\prime}}^{e_{2}}\right)$, 
	where the remainder $R^{G}$ is defined by 
	\begin{equation*}
		R_{s, t}^{G}:=\delta G_{s, t}-G_{s}^{\prime}\delta X_{s, t}, \quad \forall \left(s, t\right) \in \Delta_{2}.
	\end{equation*}
\end{definition}
Equipped with the norm 
\begin{align*}
	\left\|G, G^{\prime}\right\|_{\mathbf{D}_{X}^{\beta, \beta^{\prime}}L_{\infty}\mathcal{L}_{\gamma_{1}, \gamma_{2}}}&:=\left\|G\right\|_{E^{\beta}L_{\infty}\mathcal{L}_{\gamma_{1}, \gamma_{2}}}+\left\|G\right\|_{E^{\beta^{\prime}}L_{\infty}\mathcal{L}_{\gamma_{1}-\beta, \gamma_{2}-\beta}}\\
	&\quad+\left\|G^{\prime}\right\|_{E^{\beta^{\prime}}L_{\infty}\mathcal{L}_{\gamma_{1}, \gamma_{2}-\beta}}+\left\|\mathbb{E}_{\cdot}R^{G}\right\|_{\beta+\beta^{\prime}, \infty, \left(\gamma_{1}, \gamma_{2}-\beta-\beta^{\prime}\right)\text{-}op},
\end{align*}
$\mathbf{D}_{X}^{\beta, \beta^{\prime}}L_{\infty}\mathcal{L}_{\gamma_{1}, \gamma_{2}}\left(\mathcal{H}^{e_{1}}, \mathcal{H}^{e_{2}}\right)$ is a Banach space. 
The following result shows that the composition of a stochastic controlled rough path 
and an operator family (according to the same monotone family) is again a stochastic controlled rough path. 
\begin{proposition} \label{lc}
	Let $\left(Y, Y^{\prime}\right) \in \mathbf{D}_{X}^{\beta, \beta^{\prime}}L_{m}\mathcal{H}_{\gamma_{1}}^{e_{1}}$ 
	and $\left(G, G^{\prime}\right) \in \mathbf{D}_{X}^{\beta, \beta^{\prime}}L_{\infty}\mathcal{L}_{\gamma_{1}, \gamma_{2}}\left(\mathcal{H}^{e_{1}}, \mathcal{H}^{e_{2}}\right)$. 
	Then we have $\left(GY, GY^{\prime}+G^{\prime}Y\right) \in \mathbf{D}_{X}^{\beta, \beta^{\prime}}L_{m}\mathcal{H}_{\gamma_{2}}^{e_{2}}$ and 
	\begin{equation} \label{lce1}
		\left\|G_{0}Y_{0}\right\|_{m, \gamma_{2}} \lesssim \left\|G, G^{\prime}\right\|_{\mathbf{D}_{X}^{\beta, \beta^{\prime}}L_{\infty}\mathcal{L}_{\gamma_{1}, \gamma_{2}}}\left\|Y_{0}\right\|_{m, \gamma_{1}},
	\end{equation}
	\begin{equation} \label{lce2}
		\left\|GY, GY^{\prime}+G^{\prime}Y\right\|_{\mathbf{D}_{X}^{\beta, \beta^{\prime}}L_{m}\mathcal{H}_{\gamma_{2}}} \lesssim \left\|G, G^{\prime}\right\|_{\mathbf{D}_{X}^{\beta, \beta^{\prime}}L_{\infty}\mathcal{L}_{\gamma_{1}, \gamma_{2}}}\left\|Y, Y^{\prime}\right\|_{\mathbf{D}_{X}^{\beta, \beta^{\prime}}L_{m}\mathcal{H}_{\gamma_{1}}}.
	\end{equation}
\end{proposition}
\begin{proof}
	Clearly, we have 
	\begin{equation*}
		\left\|G_{0}Y_{0}\right\|_{m, \gamma_{2}} \lesssim \left\|G_{0}\right\|_{\infty, \left(\gamma_{1}, \gamma_{2}\right)\text{-}op}\left\|Y_{0}\right\|_{m, \gamma_{1}} \lesssim \left\|G, G^{\prime}\right\|_{\mathbf{D}_{X}^{\beta, \beta^{\prime}}L_{\infty}\mathcal{L}_{\gamma_{1}, \gamma_{2}}}\left\|Y_{0}\right\|_{m, \gamma_{1}}.
	\end{equation*}
	Since $Y \in E^{\beta}L_{m}\mathcal{H}_{\gamma_{1}}^{e_{1}}$ 
	and $G \in E^{\beta}L_{\infty}\mathcal{L}_{\gamma_{1}, \gamma_{2}}\left(\mathcal{H}^{e_{1}}, \mathcal{H}^{e_{2}}\right)$, 
	by Proposition \ref{lcp} we have $GY \in E^{\beta}L_{m}\mathcal{H}_{\gamma_{2}}^{e_{2}}$ and 
	\begin{equation*}
		\left\|GY\right\|_{E^{\beta}L_{m}\mathcal{H}_{\gamma_{2}}} \lesssim \left\|G\right\|_{E^{\beta}L_{\infty}\mathcal{L}_{\gamma_{1}, \gamma_{2}}}\left\|Y\right\|_{E^{\beta}L_{m}\mathcal{H}_{\gamma_{1}}}.
	\end{equation*}
	In a similar way, by Proposition \ref{lcp}, $Y^{\prime} \in E^{\beta^{\prime}}L_{m}\mathcal{H}_{\gamma_{1}-\beta}^{e_{1} \times e}$ and 
	$G \in E^{\beta^{\prime}}L_{\infty}\mathcal{L}_{\gamma_{1}-\beta, \gamma_{2}-\beta}\left(\mathcal{H}^{e_{1}}, \mathcal{H}^{e_{2}}\right) \subset E^{\beta^{\prime}}L_{\infty}\mathcal{L}_{\gamma_{1}-\beta, \gamma_{2}-\beta}\left(\mathcal{H}^{e_{1} \times e}, \mathcal{H}^{e_{2} \times e}\right)$ 
	ensure $GY^{\prime} \in E^{\beta^{\prime}}L_{m}\mathcal{H}_{\gamma_{2}-\beta}^{e_{2} \times e}$ and 
	\begin{equation*}
		\left\|GY^{\prime}\right\|_{E^{\beta^{\prime}}L_{m}\mathcal{H}_{\gamma_{2}-\beta}} \lesssim \left\|G\right\|_{E^{\beta^{\prime}}L_{\infty}\mathcal{L}_{\gamma_{1}-\beta, \gamma_{2}-\beta}}\left\|Y^{\prime}\right\|_{E^{\beta^{\prime}}L_{m}\mathcal{H}_{\gamma_{1}-\beta}}.
	\end{equation*}
	By Proposition \ref{in}, we have 
	$Y \in E^{\beta}L_{m}\mathcal{H}_{\gamma_{1}}^{e_{1}} \subset E^{\beta^{\prime}}L_{m}\mathcal{H}_{\gamma_{1}}^{e_{1}}$ and 
	\begin{equation*}
		\left\|Y\right\|_{E^{\beta^{\prime}}L_{m}\mathcal{H}_{\gamma_{1}}} \lesssim \left\|Y\right\|_{E^{\beta}L_{m}\mathcal{H}_{\gamma_{1}}}.
	\end{equation*}
	Combined with $G^{\prime} \in E^{\beta^{\prime}}L_{\infty}\mathcal{L}_{\gamma_{1}, \gamma_{2}-\beta}\left(\mathcal{H}^{e_{1}}, \mathcal{H}^{e_{2} \times e}\right)$, 
	by Proposition \ref{lcp} we have $G^{\prime}Y \in E^{\beta^{\prime}}L_{m}\mathcal{H}_{\gamma_{2}-\beta}^{e_{2} \times e}$ and 
	\begin{equation*}
		\left\|G^{\prime}Y\right\|_{E^{\beta^{\prime}}L_{m}\mathcal{H}_{\gamma_{2}-\beta}} \lesssim \left\|G^{\prime}\right\|_{E^{\beta^{\prime}}L_{\infty}\mathcal{L}_{\gamma_{1}, \gamma_{2}-\beta}}\left\|Y\right\|_{E^{\beta^{\prime}}L_{m}\mathcal{H}_{\gamma_{1}}} \lesssim \left\|G^{\prime}\right\|_{E^{\beta^{\prime}}L_{\infty}\mathcal{L}_{\gamma_{1}, \gamma_{2}-\beta}}\left\|Y\right\|_{E^{\beta}L_{m}\mathcal{H}_{\gamma_{1}}}.
	\end{equation*}
	Since 
	\begin{equation*}
		R_{s, t}^{GY}=G_{t}Y_{t}-G_{s}Y_{s}-\left(G_{s}Y_{s}^{\prime}+G_{s}^{\prime}Y_{s}\right)\delta X_{s, t}=\delta G_{s, t}\delta Y_{s, t}+R_{s, t}^{G}Y_{s}+G_{s}R_{s, t}^{Y}, \quad \forall \left(s, t\right) \in \Delta_{2},
	\end{equation*}
	we have 
	\begin{align*}
		\left\|\mathbb{E}_{s}R_{s, t}^{GY}\right\|_{m, \gamma_{2}-\beta-\beta^{\prime}} &\lesssim \left\|\delta G_{s, t}\right\|_{\infty, \left(\gamma_{1}-\beta, \gamma_{2}-\beta\right)\text{-}op}\left\|\delta Y_{s, t}\right\|_{m, \gamma_{1}-\beta}+\left\|\mathbb{E}_{s}R_{s, t}^{G}\right\|_{\infty, \left(\gamma_{1}, \gamma_{2}-\beta-\beta^{\prime}\right)\text{-}op}\left\|Y_{s}\right\|_{m, \gamma_{1}}\\
		&\quad+\left\|G_{s}\right\|_{\infty, \left(\gamma_{1}-\beta-\beta^{\prime}, \gamma_{2}-\beta-\beta^{\prime}\right)\text{-}op}\left\|\mathbb{E}_{s}R_{s, t}^{Y}\right\|_{m, \gamma_{1}-\beta-\beta^{\prime}}\\
		&\lesssim \left\|G, G^{\prime}\right\|_{\mathbf{D}_{X}^{\beta, \beta^{\prime}}L_{\infty}\mathcal{L}_{\gamma_{1}, \gamma_{2}}}\left\|Y, Y^{\prime}\right\|_{\mathbf{D}_{X}^{\beta, \beta^{\prime}}L_{m}\mathcal{H}_{\gamma_{1}}}\left|t-s\right|^{\beta+\beta^{\prime}}, \quad \forall \left(s, t\right) \in \Delta,
	\end{align*}
	which gives $\mathbb{E}_{\cdot}R^{GY} \in C_{2}^{\beta+\beta^{\prime}}L_{m}\mathcal{H}_{\gamma_{2}-\beta-\beta^{\prime}}^{e_{2}}$ and 
	\begin{equation*}
		\left\|\mathbb{E}_{\cdot}R^{GY}\right\|_{\beta+\beta^{\prime}, m, \gamma_{2}-\beta-\beta^{\prime}} \lesssim \left\|G, G^{\prime}\right\|_{\mathbf{D}_{X}^{\beta, \beta^{\prime}}L_{\infty}\mathcal{L}_{\gamma_{1}, \gamma_{2}}}\left\|Y, Y^{\prime}\right\|_{\mathbf{D}_{X}^{\beta, \beta^{\prime}}L_{m}\mathcal{H}_{\gamma_{1}}}.
	\end{equation*}
	Therefore, $\left(GY, GY^{\prime}+G^{\prime}Y\right) \in \mathbf{D}_{X}^{\beta, \beta^{\prime}}L_{m}\mathcal{H}_{\gamma_{2}}^{e_{2}}$ 
	and the estimate \eqref{lce2} holds. 
\end{proof}
We finish this section with the stability of composition. 
For $X, \bar{X} \in C^{\alpha}\left(\left[0, T\right], \mathbb{R}^{e}\right)$, 
$\left(G, G^{\prime}\right) \in \mathbf{D}_{X}^{\beta, \beta^{\prime}}L_{\infty}\mathcal{L}_{\gamma_{1}, \gamma_{2}}\left(\mathcal{H}^{e_{1}}, \mathcal{H}^{e_{2}}\right)$ 
and $\left(\bar{G}, \bar{G}^{\prime}\right) \in \mathbf{D}_{\bar{X}}^{\beta, \beta^{\prime}}L_{\infty}\mathcal{L}_{\gamma_{1}, \gamma_{2}}\left(\mathcal{H}^{e_{1}}, \mathcal{H}^{e_{2}}\right)$, 
we introduce the ``distance" 
\begin{align*}
	\left\|G, G^{\prime}; \bar{G}, \bar{G}^{\prime}\right\|_{\mathbf{D}_{X, \bar{X}}^{\beta, \beta^{\prime}}L_{\infty}\mathcal{L}_{\gamma_{1}, \gamma_{2}}}&:=\left\|\Delta G\right\|_{E^{\beta}L_{\infty}\mathcal{L}_{\gamma_{1}, \gamma_{2}}}+\left\|\Delta G\right\|_{E^{\beta^{\prime}}L_{\infty}\mathcal{L}_{\gamma_{1}-\beta, \gamma_{2}-\beta}}\\
	&\quad+\left\|\Delta G^{\prime}\right\|_{E^{\beta^{\prime}}L_{\infty}\mathcal{L}_{\gamma_{1}, \gamma_{2}-\beta}}+\left\|\mathbb{E}_{\cdot}\Delta R^{G}\right\|_{\beta+\beta^{\prime}, \infty, \left(\gamma_{1}, \gamma_{2}-\beta-\beta^{\prime}\right)\text{-}op}.
\end{align*}
Similarly, this ``distance" is not a metric for $X \neq \bar{X}$. 
The following stability result can be obtained analogous to the proof of 
estimates \eqref{lce1} and \eqref{lce2}. 
\begin{proposition} \label{lcs}
	Let $X, \bar{X} \in C^{\alpha}\left(\left[0, T\right], \mathbb{R}^{e}\right)$, 
	$\left(Y, Y^{\prime}\right) \in \mathbf{D}_{X}^{\beta, \beta^{\prime}}L_{m}\mathcal{H}_{\gamma}^{e_{1}}$, 
	$\left(\bar{Y}, \bar{Y}^{\prime}\right) \in \mathbf{D}_{\bar{X}}^{\beta, \beta^{\prime}}L_{m}\mathcal{H}_{\gamma}^{e_{1}}$, 
	$\left(G, G^{\prime}\right) \in \mathbf{D}_{X}^{\beta, \beta^{\prime}}L_{\infty}\mathcal{L}_{\gamma_{1}, \gamma_{2}}\left(\mathcal{H}^{e_{1}}, \mathcal{H}^{e_{2}}\right)$ 
	and $\left(\bar{G}, \bar{G}^{\prime}\right) \in \mathbf{D}_{\bar{X}}^{\beta, \beta^{\prime}}L_{\infty}\mathcal{L}_{\gamma_{1}, \gamma_{2}}\left(\mathcal{H}^{e_{1}}, \mathcal{H}^{e_{2}}\right)$. 
	Then we have 
	\begin{equation*}
		\left\|\Delta \left(G_{0}Y_{0}\right)\right\|_{m, \gamma_{2}} \lesssim \left\|G, G^{\prime}\right\|_{\mathbf{D}_{X}^{\beta, \beta^{\prime}}L_{\infty}\mathcal{L}_{\gamma_{1}, \gamma_{2}}}\left\|\Delta Y_{0}\right\|_{m, \gamma_{1}}+\left\|G, G^{\prime}; \bar{G}, \bar{G}^{\prime}\right\|_{\mathbf{D}_{X, \bar{X}}^{\beta, \beta^{\prime}}L_{\infty}\mathcal{L}_{\gamma_{1}, \gamma_{2}}}\left\|\bar{Y}_{0}\right\|_{m, \gamma_{1}},
	\end{equation*}
	\begin{align*}
		\left\|GY, GY^{\prime}+G^{\prime}Y; \bar{G}\bar{Y}, \bar{G}\bar{Y}^{\prime}+\bar{G}^{\prime}\bar{Y}\right\|_{\mathbf{D}_{X, \bar{X}}^{\beta, \beta^{\prime}}L_{m}\mathcal{H}_{\gamma_{2}}} &\lesssim \left\|G, G^{\prime}\right\|_{\mathbf{D}_{X}^{\beta, \beta^{\prime}}L_{\infty}\mathcal{L}_{\gamma_{1}, \gamma_{2}}}\left\|Y, Y^{\prime}; \bar{Y}, \bar{Y}^{\prime}\right\|_{\mathbf{D}_{X, \bar{X}}^{\beta, \beta^{\prime}}L_{m}\mathcal{H}_{\gamma_{1}}}\\
		&\quad+\left\|G, G^{\prime}; \bar{G}, \bar{G}^{\prime}\right\|_{\mathbf{D}_{X, \bar{X}}^{\beta, \beta^{\prime}}L_{\infty}\mathcal{L}_{\gamma_{1}, \gamma_{2}}}\left\|\bar{Y}, \bar{Y}^{\prime}\right\|_{\mathbf{D}_{\bar{X}}^{\beta, \beta^{\prime}}L_{m}\mathcal{H}_{\gamma_{1}}}.
	\end{align*}
\end{proposition}

\section{Mild solution of semilinear rough SPDEs}
\label{Sec4}
In this section, we fix $0 < \beta^{\prime} < \beta < \alpha$, $\gamma \in \mathbb{R}$, 
$\lambda \in \left[0, 1\right)$, $\mu \in \left[0, \frac{1}{2}\right)$, $\nu \in \left[0, \beta^{\prime}\right]$ 
and $m \in \left[2, \infty\right)$, such that $\alpha+\beta+\beta^{\prime} > 1$. 
Consider the following semilinear rough stochastic partial differential equation (rough SPDE) 
\begin{equation} \label{rspde}
	\left\{
		\begin{aligned}
			&du_{t}=\left[L_{t}u_{t}+f\left(t, u_{t}\right)\right]dt+\left(G_{t}u_{t}+g_{t}\right)d\mathbf{X}_{t}+h\left(t, u_{t}\right)dW_{t}, \quad t \in \left(0, T\right],\\
			&u_{0}=\xi.
		\end{aligned}
	\right.	
\end{equation}
Here, $\mathbf{X}=\left(X, \mathbb{X}\right) \in \mathscr{C}^{\alpha}$, 
$\xi$ is an $\mathcal{H}_{\gamma}$-valued random variable, 
$f: \left[0, T\right] \times \Omega \times \mathcal{H}_{\gamma} \rightarrow \mathcal{H}_{\gamma-\lambda}$ 
and $h: \left[0, T\right] \times \Omega \times \mathcal{H}_{\gamma} \rightarrow \mathcal{H}_{\gamma-\mu}^{d}$ 
are progressively measurable vector fields, 
$G: \left[0, T\right] \times \Omega \rightarrow \mathcal{L}\left(\mathcal{H}_{\gamma}, \mathcal{H}_{\gamma-\nu}^{e}\right)$, 
$G^{\prime}: \left[0, T\right] \times \Omega \rightarrow \mathcal{L}\left(\mathcal{H}_{\gamma}, \mathcal{H}_{\gamma-\beta-\nu}^{e \times e}\right)$ 
(implied in \eqref{rspde}), 
$g: \left[0, T\right] \times \Omega \rightarrow \mathcal{H}_{\gamma-\nu}^{e}$ and 
$g^{\prime}: \left[0, T\right] \times \Omega \rightarrow \mathcal{H}_{\gamma-\beta-\nu}^{e \times e}$ 
(also implied in \eqref{rspde}) are measurable adapted processes. 
We first introduce the following definition of mild solutions in $\mathbf{D}_{X}^{\beta, \beta^{\prime}}L_{m}\mathcal{H}_{\gamma}$. 
\begin{definition} \label{msd}
	We call $\left(u, u^{\prime}\right) \in \mathbf{D}_{X}^{\beta, \beta^{\prime}}L_{m}\mathcal{H}_{\gamma}$ 
	a mild solution of rough SPDE \eqref{rspde} if $u^{\prime}=Gu+g$, 
	$\left(Gu+g, Gu^{\prime}+G^{\prime}u+g^{\prime}\right) \in \mathbf{D}_{X}^{\beta, \beta^{\prime}}L_{m}\mathcal{H}_{\gamma-\nu}^{e}$, 
	and for every $t \in \left[0, T\right]$ and a.s. $\omega \in \Omega$, 
	\begin{equation*}
		u_{t}=S_{0, t}\xi+\int_{0}^{t}S_{r, t}f\left(r, u_{r}\right)dr+\int_{0}^{t}S_{r, t}\left(G_{r}u_{r}+g_{r}\right)d\mathbf{X}_{r}+\int_{0}^{t}S_{r, t}h\left(r, u_{r}\right)dW_{r}
	\end{equation*}
	holds in $\mathcal{H}_{\gamma}$. 
\end{definition}
Then we introduce the following assumption. 
\begin{assumption} \label{as}
	~
	\begin{enumerate}[(i)]
		\item $\mathbf{X}=\left(X, \mathbb{X}\right) \in \mathscr{C}^{\alpha}\left(\left[0, T\right], \mathbb{R}^{e}\right)$; 
		\item $\xi \in L^{m}\left(\Omega, \mathcal{F}_{0}, \mathcal{H}_{\gamma}\right)$; 
		\item $\left\|f\left(\cdot, 0\right)\right\|_{0, m, \gamma-\lambda} < \infty$ and 
		\begin{equation*}
			\left|f\left(t, u\right)-f\left(t, \bar{u}\right)\right|_{\gamma-\lambda} \lesssim \left|u-\bar{u}\right|_{\gamma}, \quad \forall t \in \left[0, T\right], \quad \forall u, \bar{u} \in \mathcal{H}_{\gamma};
		\end{equation*}
		\item $\left\|h\left(\cdot, 0\right)\right\|_{0, m, \gamma-\mu} < \infty$ and  
		\begin{equation*}
			\left|h\left(t, u\right)-h\left(t, \bar{u}\right)\right|_{\gamma-\mu} \lesssim \left|u-\bar{u}\right|_{\gamma}, \quad \forall t \in \left[0, T\right], \quad \forall u, \bar{u} \in \mathcal{H}_{\gamma}; 
		\end{equation*}
		\item $\left(G, G^{\prime}\right) \in \mathbf{D}_{X}^{\beta, \beta^{\prime}}L_{\infty}\mathcal{L}_{\gamma, \gamma-\nu}\left(\mathcal{H}, \mathcal{H}^{e}\right)$ 
		and $\left(g, g^{\prime}\right) \in \mathbf{D}_{X}^{\beta, \beta^{\prime}}L_{m}\mathcal{H}_{\gamma-\nu}^{e}$.
	\end{enumerate}
\end{assumption}
\subsection{Existence and uniqueness}
We now give the existence and uniqueness of the mild solution of \eqref{rspde}. 
\begin{theorem} \label{seu}
	Under Assumption \ref{as}, the rough SPDE \eqref{rspde} 
	has a unique mild solution $\left(u, u^{\prime}\right) \in \mathbf{D}_{X}^{\beta, \beta^{\prime}}L_{m}\mathcal{H}_{\gamma}$, and we have 
	\begin{equation} \label{se}
		\left\|u, u^{\prime}\right\|_{\mathbf{D}_{X}^{\beta, \beta^{\prime}}L_{m}\mathcal{H}_{\gamma}} \lesssim \left\|\xi\right\|_{m, \gamma}+\left\|f\left(\cdot, 0\right)\right\|_{0, m, \gamma-\lambda}+\left\|h\left(\cdot, 0\right)\right\|_{0, m, \gamma-\mu}+\left\|g, g^{\prime}\right\|_{\mathbf{D}_{X}^{\beta, \beta^{\prime}}L_{m}\mathcal{H}_{\gamma-\nu}},
	\end{equation}
	for a hidden prefactor depending on $\left|\mathbf{X}\right|_{\alpha}$ and $\left\|G, G^{\prime}\right\|_{\mathbf{D}_{X}^{\beta, \beta^{\prime}}L_{\infty}\mathcal{L}_{\gamma, \gamma-\nu}}$. 
\end{theorem}
\begin{proof}
	Let $\varepsilon \in \left(0, 1\right]$ be a constant waiting to be determined. 
	We first show the existence and uniqueness for $T \leq \varepsilon$. 
	Define 
	\begin{equation*}
		\mathbf{S}:=\left\{\left(u, u^{\prime}\right) \in \mathbf{D}_{X}^{\beta, \beta^{\prime}}L_{m}\mathcal{H}_{\gamma}: u_{0}=\xi, u_{0}^{\prime}=G_{0}\xi+g_{0}\right\}.
	\end{equation*}
	Since the process pair $t \mapsto \left(\xi+\left(G_{0}\xi+g_{0}\right)\delta X_{0, t}, G_{0}\xi+g_{0}\right)$ 
	belongs to $\mathbf{S}$, $\mathbf{S}$ is a nonempty complete metric subspace of 
	$\mathbf{D}_{X}^{\beta, \beta^{\prime}}L_{m}\mathcal{H}_{\gamma}$. 
	For every $\left(u, u^{\prime}\right) \in \mathbf{S}$, by Proposition \ref{lc} we have 
	$\left(Y, Y^{\prime}\right):=\left(Gu+g, Gu^{\prime}+G^{\prime}u+g^{\prime}\right) \in \mathbf{D}_{X}^{\beta, \beta^{\prime}}L_{m}\mathcal{H}_{\gamma-\nu}^{e}$ and 
	\begin{equation*}
		\left\|Y_{0}\right\|_{m ,\gamma-\nu} \lesssim \left\|\xi\right\|_{m, \gamma}+\left\|g_{0}\right\|_{m, \gamma-\nu}, \quad \left\|Y, Y^{\prime}\right\|_{\mathbf{D}_{X}^{\beta, \beta^{\prime}}L_{m}\mathcal{H}_{\gamma-\nu}} \lesssim \left\|u, u^{\prime}\right\|_{\mathbf{D}_{X}^{\beta, \beta^{\prime}}L_{m}\mathcal{H}_{\gamma}}+\left\|g, g^{\prime}\right\|_{\mathbf{D}_{X}^{\beta, \beta^{\prime}}L_{m}\mathcal{H}_{\gamma-\nu}}.
	\end{equation*}
	Then by Proposition \ref{rsci}, we have 
	$\left(Z, Y\right):=\left(\int_{0}^{\cdot}S_{r, \cdot}Y_{r}d\mathbf{X}_{r}, Y\right) \in \mathcal{D}_{X}^{\beta, \beta^{\prime}}L_{m}\mathcal{H}_{\gamma}$ and 
	\begin{align*}
		\left\|Z, Y\right\|_{\mathbf{D}_{X}^{\beta, \beta^{\prime}}L_{m}\mathcal{H}_{\gamma}} &\lesssim \left\|Y_{0}\right\|_{m, \gamma-\nu}+T^{\left(\alpha-\beta\right)\wedge\left(\beta-\beta^{\prime}\right)}\left\|Y, Y^{\prime}\right\|_{\mathbf{D}_{X}^{\beta, \beta^{\prime}}L_{m}\mathcal{H}_{\gamma-\nu}}\\
		&\lesssim T^{\left(\alpha-\beta\right)\wedge\left(\beta-\beta^{\prime}\right)}\left\|u, u^{\prime}\right\|_{\mathbf{D}_{X}^{\beta, \beta^{\prime}}L_{m}\mathcal{H}_{\gamma}}+\left\|\xi\right\|_{m, \gamma}+\left\|g, g^{\prime}\right\|_{\mathbf{D}_{X}^{\beta, \beta^{\prime}}L_{m}\mathcal{H}_{\gamma-\nu}}.
	\end{align*} 
	Define 
	\begin{equation*}
		F:=\int_{0}^{\cdot}S_{r, \cdot}f\left(r, u_{r}\right)dr, \quad H:=\int_{0}^{\cdot}S_{r, \cdot}h\left(r, u_{r}\right)dW_{r}, \quad 
	\end{equation*}
	\begin{equation*}
		\Phi\left(u, u^{\prime}\right)=\left(\Phi^{u}, \Phi^{u^{\prime}}\right):=\left(S_{0, \cdot}\xi+F+Z+H, Y\right).
	\end{equation*}
	Then $\Phi_{0}^{u}=\xi$ and $\Phi_{0}^{u^{\prime}}=G_{0}\xi+g_{0}$. 
	Since $S$ is bounded and strongly continuous on $\mathcal{H}_{\gamma}$, 
	we have $S_{0, \cdot}\xi \in CL_{m}\mathcal{H}_{\gamma}$ and 
	\begin{equation*}
		\left\|S_{0, \cdot}\xi\right\|_{0, m, \gamma} \lesssim \left\|\xi\right\|_{m, \gamma}. 
	\end{equation*}
	Combined with $\hat{\delta}S_{0, \cdot}\xi=0$, by Propositions \ref{ne} and \ref{de} we have 
	$\left(S_{0, \cdot}\xi, 0\right) \in \mathcal{D}_{X}^{\beta, \beta^{\prime}}L_{m}\mathcal{H}_{\gamma}$ and 
	\begin{equation*}
		\left\|S_{0, \cdot}\xi, 0\right\|_{\mathbf{D}_{X}^{\beta, \beta^{\prime}}L_{m}\mathcal{H}_{\gamma}} \lesssim \left\|\xi\right\|_{m, \gamma}. 
	\end{equation*}
	By Proposition \ref{asp} and applying Minkowski's inequality, we have 
	\begin{align*}
		\left\|F_{t}\right\|_{m, \gamma} &\lesssim \left(\mathbb{E}\left(\int_{0}^{t}\left|S_{r, t}f\left(r, u_{r}\right)\right|_{\gamma}dr\right)^{m}\right)^{\frac{1}{m}} \lesssim \int_{0}^{t}\left|t-r\right|^{-\lambda}\left\|f\left(r, u_{r}\right)\right\|_{m, \gamma-\lambda}dr\\
		&\lesssim T^{1-\lambda}\left\|u, u^{\prime}\right\|_{\mathbf{D}_{X}^{\beta, \beta^{\prime}}L_{m}\mathcal{H}_{\gamma}}+\left\|f\left(\cdot, 0\right)\right\|_{0, m, \gamma-\lambda}, \quad \forall t \in \left[0, T\right],
	\end{align*}
	\begin{align*}
		\left\|\hat{\delta}F_{s, t}\right\|_{m, \gamma-\beta-\beta^{\prime}} &\lesssim \left(\mathbb{E}\left(\int_{s}^{t}\left|S_{r, t}f\left(r, u_{r}\right)\right|_{\gamma-\beta-\beta^{\prime}}dr\right)^{m}\right)^{\frac{1}{m}} \lesssim \int_{s}^{t}\left|t-r\right|^{-\left(\lambda-\beta-\beta^{\prime}\right)^{+}}\left\|f\left(r, u_{r}\right)\right\|_{m, \gamma-\lambda}dr\\
		&\lesssim \left(T^{1-\left(\beta+\beta^{\prime}\right) \vee \lambda}\left\|u, u^{\prime}\right\|_{\mathbf{D}_{X}^{\beta, \beta^{\prime}}L_{m}\mathcal{H}_{\gamma}}+\left\|f\left(\cdot, 0\right)\right\|_{0, m, \gamma-\lambda}\right)\left|t-s\right|^{\beta+\beta^{\prime}}, \quad \forall \left(s, t\right) \in \Delta_{2}.
	\end{align*}
	Then we have $F \in CL_{m}\mathcal{H}_{\gamma} \cap \hat{C}^{\beta+\beta^{\prime}}L_{m}\mathcal{H}_{\gamma-\beta-\beta^{\prime}}$ and 
	\begin{equation*}
		\left\|F\right\|_{0, m, \gamma}+\left\|\hat{\delta}F\right\|_{\beta+\beta^{\prime}, m, \gamma-\beta-\beta^{\prime}} \lesssim T^{1-\left(\beta+\beta^{\prime}\right) \vee \lambda}\left\|u, u^{\prime}\right\|_{\mathbf{D}_{X}^{\beta, \beta^{\prime}}L_{m}\mathcal{H}_{\gamma}}+\left\|f\left(\cdot, 0\right)\right\|_{0, m, \gamma-\lambda}.
	\end{equation*}
	By Propositions \ref{in} and \ref{ne}, we have $F \in E^{\beta+\beta^{\prime}}L_{m}\mathcal{H}_{\gamma} \subset E^{\beta}L_{m}\mathcal{H}_{\gamma}$ and 
	\begin{equation*}
		\left\|F\right\|_{E^{\beta}L_{m}\mathcal{H}_{\gamma}} \lesssim \left\|F\right\|_{E^{\beta+\beta^{\prime}}L_{m}\mathcal{H}_{\gamma}} \lesssim T^{1-\left(\beta+\beta^{\prime}\right) \vee \lambda}\left\|u, u^{\prime}\right\|_{\mathbf{D}_{X}^{\beta, \beta^{\prime}}L_{m}\mathcal{H}_{\gamma}}+\left\|f\left(\cdot, 0\right)\right\|_{0, m, \gamma-\lambda}. 
	\end{equation*}
	Then $\left(F, 0\right) \in \mathcal{D}_{X}^{\beta, \beta^{\prime}}L_{m}\mathcal{H}_{\gamma}$ 
	since $R^{F}=\delta F \in C_{2}^{\beta+\beta^{\prime}}L_{m}\mathcal{H}_{\gamma-\beta-\beta^{\prime}}$, and we have 
	\begin{align*}
		\left\|F, 0\right\|_{\mathbf{D}_{X}^{\beta, \beta^{\prime}}L_{m}\mathcal{H}_{\gamma}} &\lesssim \left\|F\right\|_{E^{\beta}L_{m}\mathcal{H}_{\gamma}}+\left\|\delta F\right\|_{\beta+\beta^{\prime}, m, \gamma-\beta-\beta^{\prime}}\\
		&\lesssim T^{1-\left(\beta+\beta^{\prime}\right) \vee \lambda}\left\|u, u^{\prime}\right\|_{\mathbf{D}_{X}^{\beta, \beta^{\prime}}L_{m}\mathcal{H}_{\gamma}}+\left\|f\left(\cdot, 0\right)\right\|_{0, m, \gamma-\lambda}. 
	\end{align*}
	In a similar way, by Proposition \ref{asp} and applying Minkowski's inequality and the BDG inequality, we have 
	\begin{align*}
		\left\|H_{t}\right\|_{m, \gamma} &\lesssim \left(\mathbb{E}\left(\int_{0}^{t}\left|S_{r, t}h\left(r, u_{r}\right)\right|_{\gamma}^{2}dr\right)^{\frac{m}{2}}\right)^{\frac{1}{m}} \lesssim \left(\int_{0}^{t}\left|t-r\right|^{-2\mu}\left\|h\left(r, u_{r}\right)\right\|_{m, \gamma-\mu}^{2}dr\right)^{\frac{1}{2}}\\
		&\lesssim T^{\frac{1}{2}-\mu}\left\|u, u^{\prime}\right\|_{\mathbf{D}_{X}^{\beta, \beta^{\prime}}L_{m}\mathcal{H}_{\gamma}}+\left\|h\left(\cdot, 0\right)\right\|_{0, m, \gamma-\mu}, \quad \forall t \in \left[0, T\right],
	\end{align*}
	\begin{align*}
		\left\|\hat{\delta}H_{s, t}\right\|_{m, \gamma-\beta} &\lesssim \left(\mathbb{E}\left(\int_{s}^{t}\left|S_{r, t}h\left(r, u_{r}\right)\right|_{\gamma-\beta}^{2}dr\right)^{\frac{m}{2}}\right)^{\frac{1}{m}} \lesssim \left(\int_{s}^{t}\left|t-r\right|^{-2\left(\mu-\beta\right)^{+}}\left\|h\left(r, u_{r}\right)\right\|_{m, \gamma-\mu}^{2}dr\right)^{\frac{1}{2}}\\
		&\lesssim \left(T^{\frac{1}{2}-\beta \vee \mu}\left\|u, u^{\prime}\right\|_{\mathbf{D}_{X}^{\beta, \beta^{\prime}}L_{m}\mathcal{H}_{\gamma}}+\left\|h\left(\cdot, 0\right)\right\|_{0, m, \gamma-\mu}\right)\left|t-s\right|^{\beta}, \quad \forall \left(s, t\right) \in \Delta_{2}.
	\end{align*}
	Then we have $H \in CL_{m}\mathcal{H}_{\gamma} \cap \hat{C}^{\beta}L_{m}\mathcal{H}_{\gamma-\beta}$ and 
	\begin{equation*}
		\left\|H\right\|_{0, m, \gamma}+\left\|\hat{\delta}H\right\|_{\beta, m, \gamma-\beta} \lesssim T^{\frac{1}{2}-\beta \vee \mu}\left\|u, u^{\prime}\right\|_{\mathbf{D}_{X}^{\beta, \beta^{\prime}}L_{m}\mathcal{H}_{\gamma}}+\left\|h\left(\cdot, 0\right)\right\|_{0, m, \gamma-\mu}.
	\end{equation*}
	By Proposition \ref{ne}, we have $H \in E^{\beta}L_{m}\mathcal{H}_{\gamma}$ and 
	\begin{equation*}
		\left\|H\right\|_{E^{\beta}L_{m}\mathcal{H}_{\gamma}} \lesssim T^{\frac{1}{2}-\beta \vee \mu}\left\|u, u^{\prime}\right\|_{\mathbf{D}_{X}^{\beta, \beta^{\prime}}L_{m}\mathcal{H}_{\gamma}}+\left\|h\left(\cdot, 0\right)\right\|_{0, m, \gamma-\mu}.
	\end{equation*}
	Combined with $\mathbb{E}_{\cdot}\hat{\delta}H=0$, by Proposition \ref{de} we have 
	$\left(H, 0\right) \in \mathbf{D}_{X}^{\beta, \beta^{\prime}}L_{m}\mathcal{H}_{\gamma}$ and 
	\begin{equation*}
		\left\|H, 0\right\|_{\mathbf{D}_{X}^{\beta, \beta^{\prime}}L_{m}\mathcal{H}_{\gamma}} \lesssim \left\|H\right\|_{E^{\beta}L_{m}\mathcal{H}_{\gamma}} \lesssim T^{\frac{1}{2}-\beta \vee \mu}\left\|u, u^{\prime}\right\|_{\mathbf{D}_{X}^{\beta, \beta^{\prime}}L_{m}\mathcal{H}_{\gamma}}+\left\|h\left(\cdot, 0\right)\right\|_{0, m, \gamma-\mu}.
	\end{equation*}
	Therefore, $\Phi\left(u, u^{\prime}\right) \in \mathbf{S}$ and there exists $\sigma > 0$ only depending on $\alpha, \beta, \beta^{\prime}, \lambda, \mu$ and $\nu$ such that 
	\begin{align}
		\left\|\Phi^{u}, \Phi^{u^{\prime}}\right\|_{\mathbf{D}_{X}^{\beta, \beta^{\prime}}L_{m}\mathcal{H}_{\gamma}} &\lesssim T^{\sigma}\left\|u, u^{\prime}\right\|_{\mathbf{D}_{X}^{\beta, \beta^{\prime}}L_{m}\mathcal{H}_{\gamma}}+\left\|\xi\right\|_{m, \gamma}+\left\|f\left(\cdot, 0\right)\right\|_{0, m, \gamma-\lambda} \notag\\
		&\quad+\left\|h\left(\cdot, 0\right)\right\|_{0, m, \gamma-\mu}+\left\|g, g^{\prime}\right\|_{\mathbf{D}_{X}^{\beta, \beta^{\prime}}L_{m}\mathcal{H}_{\gamma-\nu}}. \label{phie}
	\end{align}
	For any other $\left(\bar{u}, \bar{u}^{\prime}\right) \in \mathbf{S}$, note that 
	\begin{equation*}
		\Delta \Phi^{u}=\int_{0}^{\cdot}S_{r, \cdot}\Delta f\left(r, u_{r}\right)dr+\int_{0}^{\cdot}S_{r, \cdot}G_{r}\Delta u_{r}d\mathbf{X}_{r}+\int_{0}^{\cdot}S_{r, \cdot}\Delta h\left(r, u_{r}\right)dW_{r}, \quad \Delta \Phi^{u^{\prime}}=G\Delta u.
	\end{equation*}
	Analogous to the above arguments, we have 
	\begin{equation*}
		\left\|\Delta \Phi^{u}, \Delta \Phi^{u^{\prime}}\right\|_{\mathbf{D}_{X}^{\beta, \beta^{\prime}}L_{m}\mathcal{H}_{\gamma}} \lesssim T^{\sigma}\left\|\Delta u, \Delta u^{\prime}\right\|_{\mathbf{D}_{X}^{\beta, \beta^{\prime}}L_{m}\mathcal{H}_{\gamma}}.
	\end{equation*}
	Hence, we can choose $\varepsilon \in \left(0, 1\right]$ such that 
	$\Phi$ is a contraction map in $\mathbf{S}$ for $T \leq \varepsilon$. 
	Applying the Banach fixed-point theorem, $\Phi$ has a unique fixed point $\left(u, u^{\prime}\right)$ in $\mathbf{S}$, 
	which is the unique mild solution of \eqref{rspde}.\\
	\indent
	For arbitrary $T$, consider a partition $0=t_{0}<\cdots<t_{N}=T$ such that 
	$t_{i+1}-t_{i} \leq \varepsilon$ for $i=0, \cdots, N-1$. 
	Define $\left(u_{0}, u_{0}^{\prime}\right)=\left(u_{t_{0}}^{0}, u_{t_{0}}^{0, \prime}\right):=\left(\xi, G_{0}\xi+g_{0}\right)$, 
	and then define $\left(u^{i}, u^{i, \prime}\right)$ recursively on 
	$\left(t_{i-1},t_{i}\right]$ for $i=1, \cdots, N$ by the mild solution 
	in $\mathbf{D}_{X}^{\beta, \beta^{\prime}}L_{m}\mathcal{H}_{\gamma}\left[t_{i-1}, t_{i}\right]$ of the rough SPDE 
	\begin{equation*}
		\left\{
			\begin{aligned}
				&du_{t}^{i}=\left[L_{t}u_{t}^{i}+f\left(t, u_{t}^{i}\right)\right]dt+\left(G_{t}u_{t}^{i}+g_{t}\right)d\mathbf{X}_{t}+h\left(t, u_{t}^{i}\right)dW_{t}, \quad t \in \left(t_{i-1}, t_{i}\right],\\
				&u_{t_{i-1}}^{i}=u_{t_{i-1}}^{i-1}.
			\end{aligned}
		\right.	
	\end{equation*}
	Analogous to the above arguments but replacing $\xi$ by $u^{i-1}_{t_{i-1}}$, we can get 
	the existence and uniqueness of $\left(u^{i}, u^{i, \prime}\right)$, since $\varepsilon$ does not depend on $\xi$. 
	Define $\left(u_{t}, u_{t}^{\prime}\right):=\left(u_{t}^{i}, u_{t}^{i, \prime}\right)$ for every 
	$t \in \left(t_{i-1},t_{i}\right]$ and $i=1, \cdots, N$. 
	Then $\left(u, u^{\prime}\right) \in \mathbf{D}_{X}^{\beta, \beta^{\prime}}L_{m}\mathcal{H}_{\gamma}\left[0, T\right]$ is the unique mild solution of \eqref{rspde}.\\
	\indent
	At last, we show the estimate \eqref{se}. For $T \leq \varepsilon$, 
	the estimate \eqref{se} is implied by \eqref{phie}. 
	For arbitrary $T$, we similarly have 
	\begin{equation*}
		\left\|u^{i}, u^{i, \prime}\right\|_{\mathbf{D}_{X}^{\beta, \beta^{\prime}}L_{m}\mathcal{H}_{\gamma}\left[t_{i-1}, t_{i}\right]} \lesssim \left\|u_{t_{i-1}}^{i-1}\right\|_{m, \gamma}+\left\|f\left(\cdot, 0\right)\right\|_{0, m, \gamma-\lambda}+\left\|h\left(\cdot, 0\right)\right\|_{0, m, \gamma-\mu}+\left\|g, g^{\prime}\right\|_{\mathbf{D}_{X}^{\beta, \beta^{\prime}}L_{m}\mathcal{H}_{\gamma-\nu}},
	\end{equation*}
	for every $i=1, 2, \cdots, N$. By induction, we get the estimate \eqref{se}. 
\end{proof}
\subsection{Continuity of the mild solution map}
Next, we show the continuity of the mild solution map. 
\begin{theorem} \label{csm}
	Let $\left(\mathbf{X}, \xi, f, h, G, G^{\prime}, g, g^{\prime}\right)$ and 
	$\left(\bar{\mathbf{X}}, \bar{\xi}, f, h, \bar{G}, \bar{G}^{\prime}, \bar{g}, \bar{g}^{\prime}\right)$
	satisfy Assumption \ref{as}, and $\left(u, u^{\prime}\right) \in \mathbf{D}_{X}^{\beta, \beta^{\prime}}L_{m}\mathcal{H}_{\gamma}$, 
	$\left(\bar{u}, \bar{u}^{\prime}\right) \in \mathbf{D}_{\bar{X}}^{\beta, \beta^{\prime}}L_{m}\mathcal{H}_{\gamma}$ 
	be the mild solution of the corresponding rough SPDE \eqref{rspde}. 
	Assume there exists $R \geq 0$ such that 
	\begin{align*}
		\left|\mathbf{X}\right|_{\alpha}, \left|\bar{\mathbf{X}}\right|_{\alpha}, \left\|\xi\right\|_{m, \gamma}, \left\|\bar{\xi}\right\|_{m, \gamma}, \left\|f\left(\cdot, 0\right)\right\|_{0, m, \gamma-\lambda}, \left\|h\left(\cdot, 0\right)\right\|_{0, m, \gamma-\mu} &\leq R,\\
		\left\|G, G^{\prime}\right\|_{\mathbf{D}_{X}^{\beta, \beta^{\prime}}L_{\infty}\mathcal{L}_{\gamma, \gamma-\nu}}, \left\|\bar{G}, \bar{G}^{\prime}\right\|_{\mathbf{D}_{\bar{X}}^{\beta, \beta^{\prime}}L_{\infty}\mathcal{L}_{\gamma, \gamma-\nu}}, \left\|g, g^{\prime}\right\|_{\mathbf{D}_{X}^{\beta, \beta^{\prime}}L_{m}\mathcal{H}_{\gamma-\nu}}, \left\|\bar{g}, \bar{g}^{\prime}\right\|_{\mathbf{D}_{\bar{X}}^{\beta, \beta^{\prime}}L_{m}\mathcal{H}_{\gamma-\nu}} &\leq R.
	\end{align*}
	Then we have 
	\begin{align}
		\left\|u, u^{\prime}; \bar{u}, \bar{u}^{\prime}\right\|_{\mathbf{D}_{X, \bar{X}}^{\beta, \beta^{\prime}}L_{m}\mathcal{H}_{\gamma}} &\lesssim \rho_{\alpha}\left(\mathbf{X}, \bar{\mathbf{X}}\right)+\left\|\Delta \xi\right\|_{m, \gamma}+\left\|G, G^{\prime}; \bar{G}, \bar{G}^{\prime}\right\|_{\mathbf{D}_{X, \bar{X}}^{\beta, \beta^{\prime}}L_{\infty}\mathcal{L}_{\gamma, \gamma-\nu}} \notag\\
		&\quad+\left\|g, g^{\prime}; \bar{g}, \bar{g}^{\prime}\right\|_{\mathbf{D}_{X, \bar{X}}^{\beta, \beta^{\prime}}L_{m}\mathcal{H}_{\gamma-\nu}}, \label{csme}
	\end{align}
	for a hidden prefactor depending on $R$. 
\end{theorem}
\begin{proof}
	Recall the definition of $Y, Y^{\prime}, Z, F$ and $H$ in the proof of Theorem \ref{seu}. 
	We define $\bar{Y}, \bar{Y}^{\prime}, \bar{Z}, \bar{F}$ and $\bar{H}$ in a similar way. 
	By Theorem \ref{seu}, we have 
	\begin{equation*}
		\left\|u, u^{\prime}\right\|_{\mathbf{D}_{X}^{\beta, \beta^{\prime}}L_{m}\mathcal{H}_{\gamma}}, \left\|\bar{u}, \bar{u}^{\prime}\right\|_{\mathbf{D}_{\bar{X}}^{\beta, \beta^{\prime}}L_{m}\mathcal{H}_{\gamma}} \leq C_{R}.
	\end{equation*}
	Then by Propositions \ref{rscs} and \ref{lcs}, we have 
	\begin{align*}
		\left\|Z, Y; \bar{Z}, \bar{Y}\right\|_{\mathbf{D}_{X, \bar{X}}^{\beta, \beta^{\prime}}L_{m}\mathcal{H}_{\gamma}} &\lesssim \rho_{\alpha}\left(\mathbf{X}, \bar{\mathbf{X}}\right)+\left\|\Delta Y_{0}\right\|_{m, \gamma-\nu}+T^{\left(\alpha-\beta\right)\wedge\left(\beta-\beta^{\prime}\right)}\left\|Y, Y^{\prime}; \bar{Y}, \bar{Y}^{\prime}\right\|_{\mathbf{D}_{X, \bar{X}}^{\beta, \beta^{\prime}}L_{m}\mathcal{H}_{\gamma-\nu}}\\
		&\lesssim T^{\left(\alpha-\beta\right)\wedge\left(\beta-\beta^{\prime}\right)}\left\|u, u^{\prime}; \bar{u}, \bar{u}^{\prime}\right\|_{\mathbf{D}_{X, \bar{X}}^{\beta, \beta^{\prime}}L_{m}\mathcal{H}_{\gamma}}+\rho_{\alpha}\left(\mathbf{X}, \bar{\mathbf{X}}\right)+\left\|\Delta \xi\right\|_{m, \gamma}\\
		&\quad+\left\|G, G^{\prime}; \bar{G}, \bar{G}^{\prime}\right\|_{\mathbf{D}_{X, \bar{X}}^{\beta, \beta^{\prime}}L_{\infty}\mathcal{L}_{\gamma, \gamma-\nu}}+\left\|g, g^{\prime}; \bar{g}, \bar{g}^{\prime}\right\|_{\mathbf{D}_{X, \bar{X}}^{\beta, \beta^{\prime}}L_{m}\mathcal{H}_{\gamma-\nu}}.
	\end{align*}
	Analogous to the proof of Theorem \ref{seu}, we have 
	\begin{equation*}
		\left\|S_{0, \cdot}\xi, 0; S_{0, \cdot}\bar{\xi}, 0\right\|_{\mathbf{D}_{X, \bar{X}}^{\beta, \beta^{\prime}}L_{m}\mathcal{H}_{\gamma}} \lesssim \left\|\Delta \xi\right\|_{m, \gamma},
	\end{equation*}
	\begin{equation*}
		\left\|F, 0; \bar{F}, 0\right\|_{\mathbf{D}_{X, \bar{X}}^{\beta, \beta^{\prime}}L_{m}\mathcal{H}_{\gamma}} \lesssim T^{1-\left(\beta+\beta^{\prime}\right) \vee \lambda}\left\|u, u^{\prime}; \bar{u}, \bar{u}^{\prime}\right\|_{\mathbf{D}_{X, \bar{X}}^{\beta, \beta^{\prime}}L_{m}\mathcal{H}_{\gamma}},
	\end{equation*}
	\begin{equation*}
		\left\|H, 0; \bar{H}, 0\right\|_{\mathbf{D}_{X, \bar{X}}^{\beta, \beta^{\prime}}L_{m}\mathcal{H}_{\gamma}} \lesssim T^{\frac{1}{2}-\beta \vee \mu}\left\|u, u^{\prime}; \bar{u}, \bar{u}^{\prime}\right\|_{\mathbf{D}_{X, \bar{X}}^{\beta, \beta^{\prime}}L_{m}\mathcal{H}_{\gamma}}.
	\end{equation*}
	Since $\Delta u=S_{0, \cdot}\Delta \xi+\Delta F+\Delta Z+\Delta H$ and $\Delta u^{\prime}=\Delta Y$, 
	there exists $\sigma > 0$ only depending on $\alpha, \beta, \beta^{\prime}, \lambda, \mu$ and $\nu$ such that 
	\begin{align*}
		\left\|u, u^{\prime}; \bar{u}, \bar{u}^{\prime}\right\|_{\mathbf{D}_{X, \bar{X}}^{\beta, \beta^{\prime}}L_{m}\mathcal{H}_{\gamma}} &\lesssim T^{\sigma}\left\|u, u^{\prime}; \bar{u}, \bar{u}^{\prime}\right\|_{\mathbf{D}_{X, \bar{X}}^{\beta, \beta^{\prime}}L_{m}\mathcal{H}_{\gamma}}+\rho_{\alpha}\left(\mathbf{X}, \bar{\mathbf{X}}\right)+\left\|\Delta \xi\right\|_{m, \gamma}\\
		&\quad+\left\|G, G^{\prime}; \bar{G}, \bar{G}^{\prime}\right\|_{\mathbf{D}_{X, \bar{X}}^{\beta, \beta^{\prime}}L_{\infty}\mathcal{L}_{\gamma, \gamma-\nu}}+\left\|g, g^{\prime}; \bar{g}, \bar{g}^{\prime}\right\|_{\mathbf{D}_{X, \bar{X}}^{\beta, \beta^{\prime}}L_{m}\mathcal{H}_{\gamma-\nu}} \label{cms}
	\end{align*}
	Hence, for $T$ sufficiently small we get the estimate \eqref{csme}. 
	The general result can be obtained by induction. 
\end{proof}
\subsection{Spatial regularity of the mild solution away from the initial time}
We show that the mild solution of \eqref{rspde} is regularized in the space at $t > 0$. 
\begin{proposition} \label{sr}
	Let Assumption \ref{as} hold and $\left(u, u^{\prime}\right) \in \mathbf{D}_{X}^{\beta, \beta^{\prime}}L_{m}\mathcal{H}_{\gamma}$ 
	be the mild solution of \eqref{rspde}. 
	Then we have $u \in \hat{C}^{\left(1-\lambda\right) \wedge \left(\frac{1}{2}-\mu\right) \wedge \left(\alpha-\nu\right)-\theta}L_{m}\mathcal{H}_{\gamma+\theta}\left[t, T\right]$ 
	for every $t \in \left(0, T\right]$ and $0 \leq \theta < \left(1-\lambda\right) \wedge \left(\frac{1}{2}-\mu\right) \wedge \left(\alpha-\nu\right)$, and 
	\begin{equation} \label{sre}
		\left\|u_{t}\right\|_{m, \gamma+\theta} \lesssim t^{-\theta}\left\|\xi\right\|_{m, \gamma}+\left\|f\left(\cdot, 0\right)\right\|_{0, m, \gamma-\lambda}+\left\|h\left(\cdot, 0\right)\right\|_{0, m, \gamma-\mu}+\left\|g, g^{\prime}\right\|_{\mathbf{D}_{X}^{\beta, \beta^{\prime}}L_{m}\mathcal{H}_{\gamma-\nu}},
	\end{equation}
	for a hidden prefactor depending on $\left|\mathbf{X}\right|_{\alpha}$ and $\left\|G, G^{\prime}\right\|_{\mathbf{D}_{X}^{\beta, \beta^{\prime}}L_{\infty}\mathcal{L}_{\gamma, \gamma-\nu}}$. 
\end{proposition}
\begin{proof}
	Recall the definition of $Y, Y^{\prime}, Z, F$ and $H$ in the proof of Theorem \ref{seu}. 
	For every fixed $0 \leq \theta < \left(1-\lambda\right) \wedge \left(\frac{1}{2}-\mu\right) \wedge \left(\alpha-\nu\right)$, 
	by Propositions \ref{rsci} and \ref{lc} we have $Z \in \hat{C}^{\alpha-\nu-\theta}L_{m}\mathcal{H}_{\gamma+\theta}$ and 
	\begin{equation*}
		\left\|\hat{\delta}Z\right\|_{\alpha-\nu-\theta, m, \gamma+\theta} \lesssim \left\|Y, Y^{\prime}\right\|_{\mathbf{D}_{X}^{\beta, \beta^{\prime}}L_{m}\mathcal{H}_{\gamma-\nu}} \lesssim \left\|u, u^{\prime}\right\|_{\mathbf{D}_{X}^{\beta, \beta^{\prime}}L_{m}\mathcal{H}_{\gamma}}+\left\|g, g^{\prime}\right\|_{\mathbf{D}_{X}^{\beta, \beta^{\prime}}L_{m}\mathcal{H}_{\gamma-\nu}}. 
	\end{equation*}
	By Proposition \ref{asp}, we have 
	\begin{equation*}
		\left\|S_{0, t}\xi\right\|_{m, \gamma+\theta} \lesssim t^{-\theta}\left\|\xi\right\|_{m, \gamma}, \quad \forall t \in \left(0, T\right].
	\end{equation*}
	Since $S$ is strongly continuous on $\mathcal{H}_{\gamma+\theta}$ and $\hat{\delta}S_{0, \cdot}\xi=0$, 
	we have $S_{0, \cdot}\xi \in \hat{C}^{\sigma}L_{m}\mathcal{H}_{\gamma+\theta}\left[t, T\right]$ 
	for every $t \in \left(0, T\right]$ and $\sigma \in \left(0, 1\right)$. 
	By Proposition \ref{asp} and applying Minkowski's inequality, we have 
	\begin{align*}
		\left\|\hat{\delta}F_{s, t}\right\|_{m, \gamma+\theta} &\lesssim \left(\mathbb{E}\left(\int_{s}^{t}\left|S_{r, t}f\left(r, u_{r}\right)\right|_{\gamma+\theta}dr\right)^{m}\right)^{\frac{1}{m}} \lesssim \int_{s}^{t}\left(t-s\right)^{-\left(\lambda+\theta\right)}\left\|f\left(r, u_{r}\right)\right\|_{m, \gamma-\lambda}dr\\
		&\lesssim \left(\left\|u, u^{\prime}\right\|_{\mathbf{D}_{X}^{\beta, \beta^{\prime}}L_{m}\mathcal{H}_{\gamma}}+\left\|f\left(\cdot, 0\right)\right\|_{0, m, \gamma-\lambda}\right)\left|t-s\right|^{1-\lambda-\theta}, \quad \forall \left(s, t\right) \in \Delta_{2},
	\end{align*}
	which gives $F \in \hat{C}^{1-\lambda-\theta}L_{m}\mathcal{H}_{\gamma+\theta}$. 
	In a similar way, by Proposition \ref{asp} and applying Minkowski's inequality and the BDG inequality, we have 
	\begin{align*}
		\left\|\hat{\delta}H_{s, t}\right\|_{m, \gamma+\theta} &\lesssim \left(\mathbb{E}\left(\int_{s}^{t}\left|S_{r, t}h\left(r, u_{r}\right)\right|_{\gamma+\theta}^{2}dr\right)^{\frac{m}{2}}\right)^{\frac{1}{m}} \lesssim \left(\int_{s}^{t}\left(t-r\right)^{-2\left(\mu+\theta\right)}\left\|h\left(r, u_{r}\right)\right\|_{m, \gamma-\mu}^{2}dr\right)^{\frac{1}{2}}\\
		&\lesssim \left(\left\|u, u^{\prime}\right\|_{\mathbf{D}_{X}^{\beta, \beta^{\prime}}L_{m}\mathcal{H}_{\gamma}}+\left\|h\left(\cdot, 0\right)\right\|_{0, m, \gamma-\mu}\right)\left|t-s\right|^{\frac{1}{2}-\mu-\theta}, \quad \forall \left(s, t\right) \in \Delta_{2},
	\end{align*}
	which gives $H \in \hat{C}^{\frac{1}{2}-\mu-\theta}L_{m}\mathcal{H}_{\gamma+\theta}$. 
	Therefore, we have $u \in \hat{C}^{\left(1-\lambda\right) \wedge \left(\frac{1}{2}-\mu\right) \wedge \left(\alpha-\nu\right)-\theta}L_{m}\mathcal{H}_{\gamma+\theta}\left[t, T\right]$ 
	for every $t \in \left(0, T\right]$, and 
	\begin{align*}
		\left\|u_{t}\right\|_{m, \gamma+\theta} &\leq \left\|S_{0, t}\xi\right\|_{m, \gamma+\theta}+\left\|\hat{\delta}Z_{0, t}\right\|_{m, \gamma+\theta}+\left\|\hat{\delta}F_{0, t}\right\|_{m, \gamma+\theta}+\left\|\hat{\delta}H_{0, t}\right\|_{m, \gamma+\theta}\\
		&\lesssim \left\|u, u^{\prime}\right\|_{\mathbf{D}_{X}^{\beta, \beta^{\prime}}L_{m}\mathcal{H}_{\gamma}}+t^{-\theta}\left\|\xi\right\|_{m, \gamma}+\left\|f\left(\cdot, 0\right)\right\|_{0, m, \gamma-\lambda}\\
		&\quad+\left\|h\left(\cdot, 0\right)\right\|_{0, m, \gamma-\mu}+\left\|g, g^{\prime}\right\|_{\mathbf{D}_{X}^{\beta, \beta^{\prime}}L_{m}\mathcal{H}_{\gamma-\nu}}.
	\end{align*}
	In view of \eqref{se}, we get the estimate \eqref{sre}. 
\end{proof}
\begin{remark}
	When $h \equiv 0$, our results coincide with those of 
	Gerasimovičs et al. \cite{GHN21} and Hesse and Neamţu \cite{HesN22}, 
	and moreover our coefficients are allowed to be time-varying. 
	However, our approach fails to apply in a straightforward way to SPDEs with a nonlinear rough term 
	\begin{equation} \label{grspde}
		\left\{
			\begin{aligned}
				&du_{t}=\left[L_{t}u_{t}+f\left(t, u_{t}\right)\right]dt+g\left(t, u_{t}\right)d\mathbf{X}_{t}+h\left(t, u_{t}\right)dW_{t}, \quad t \in \left(0, T\right],\\
				&u_{0}=\xi.
			\end{aligned}
		\right.	
	\end{equation}
	The difficulty arises from the fact that the composition of a stochastic controlled rough path $\left(Y, Y^{\prime}\right)$ 
	with a nonlinear (even deterministic and time-invariant) vector field $g$, 
	as a stochastic controlled rough path, might not stay like the linear rough term 
	within the space of the same integrability. 
	More precisely, both $\left(Y, Y^{\prime}\right) \in \mathbf{D}_{X}^{\beta, \beta^{\prime}}L_{m}\mathcal{H}_{\gamma}^{e_{1}}$ 
	and $g \in \cap_{\gamma \in \mathbb{R}} C_{b}^{3}\left(\mathcal{H}_{\gamma}^{e_{1}}, \mathcal{H}_{\gamma}^{e_{2}}\right)$ 
	only ensure $\left(g\left(Y\right), Dg\left(Y\right)Y^{\prime}\right) \in \mathbf{D}_{X}^{\beta, \beta^{\prime}}L_{\frac{m}{2}}\mathcal{H}_{\gamma}^{e_{2}}$.\\
	\indent
	In contrast, rough SDEs in Friz et al. \cite{FHL21} 
	\begin{equation*} 
		Y_{t}=\xi+\int_{0}^{t} f\left(t, Y_{t}\right) d r+\int_{0}^{t} g\left(t, Y_{t}\right)d\mathbf{X}_{t}+\int_{0}^{t} h\left(t, Y_{t}\right)dW_{t}, \quad t \in \left[0, T\right],
	\end{equation*}
	are discussed with the space of stochastic controlled rough paths of $\left(m, \infty\right)$-integrability 
	$\mathbf{D}_{X}^{\beta, \beta^{\prime}}L_{m, \infty}$ (see \cite{FHL21}*{Definition 3.2}), 
	and both $\left(Y, Y^{\prime}\right) \in \mathbf{D}_{X}^{\beta, \beta^{\prime}}L_{m, \infty}$ and $g \in C_{b}^{3}$ 
	ensure $\left(g\left(Y\right), Dg\left(Y\right)Y^{\prime}\right) \in \mathbf{D}_{X}^{\beta, \beta^{\prime}}L_{m, \infty}$ (see \cite{FHL21}*{Lemma 3.13}). 
	However, this method still fails to apply in a straightforward way to rough SPDE \eqref{grspde}. 
	Indeed, if we combine the definition of spaces $\mathbf{D}_{X}^{\beta, \beta^{\prime}}L_{m, \infty}$ and $\mathbf{D}_{X}^{\beta, \beta^{\prime}}L_{m}\mathcal{H}_{\gamma}^{e_{1}}$, 
	defined in \cite{FHL21}*{Definition 3.2} and Definition \ref{scrp} respectively, 
	to define the space of $\left(m, \infty\right)$-integrable stochastic controlled rough paths 
	according to $\mathcal{H}_{\gamma}^{e_{1}}$, 
	we cannot get its equivalent definition like space $\mathbf{D}_{X}^{\beta, \beta^{\prime}}L_{m}\mathcal{H}_{\gamma}^{e_{1}}$ in Proposition \ref{de}, 
	since the process $\mathbb{E}_{\cdot}R^{Y}-\mathbb{E}_{\cdot}\hat{R}^{Y}$ may not be essentially bounded. 
\end{remark}

\section{An example}
\label{Sec5}
Let $\mathbb{T}^{n}$ be the $n$-dimensional torus and 
$H^{\gamma}\left(\mathbb{T}^{n}\right):=H^{\gamma, 2}\left(\mathbb{T}^{n}, \mathbb{R}\right)$ 
be the $L^{2}$-based Bessel potential space for $\gamma \in \mathbb{R}$. 
Define $\mathcal{H}_{\gamma}:=H^{2\gamma}\left(\mathbb{T}^{n}\right)$. 
Then $\left(\mathcal{H}_{\gamma}\right)_{\gamma \in \mathbb{R}}$ is a monotone family of interpolation Hilbert spaces 
(see \cite[Chap. 16]{Yag10} and \cite{Lun18} for properties of Bessel potential spaces and interpolation spaces). 
Let $a \in C^{\sigma, \infty}\left(\left[0, T\right] \times \mathbb{T}^{n}, \mathbb{R}^{n \times n}\right)$ 
for some $\sigma \in \left(0, 1\right)$. 
Assume $a$ satisfies the uniform ellipticity condition 
\begin{equation*}
	v^{\top}a_{t}\left(x\right)v \geq K\left|v\right|^{2}, \quad \forall \left(t, x\right) \in \left[0, T\right] \times \mathbb{T}^{n}, \quad \forall v \in \mathbb{R}^{n},
\end{equation*}
for some $K > 0$. 
Define 
\begin{equation*}
	L_{t}u\left(x\right):=\nabla \cdot \left(a_{t}\left(x\right)\nabla u\left(x\right)\right), \quad \forall \left(t, x\right) \in \left[0, T\right] \times \mathbb{T}^{n}, \quad \forall u \in H^{-\infty}\left(\mathbb{T}^{n}\right).
\end{equation*}
Then for every $t \in \left[0, T\right]$ and $\gamma \in \mathbb{R}$, 
$L_{t}$ is a closed densely defined linear operator on $H^{2\gamma}\left(\mathbb{T}^{n}\right)$ 
and its domain contains $H^{2\gamma+2}\left(\mathbb{T}^{n}\right)$, 
and further the operator family $\left(L_{t}\right)_{t \in \left[0, T\right]}$ generates a propagator 
$S: \Delta_{2} \rightarrow \cap_{\gamma \in \mathbb{R}}\mathcal{L}\left(H^{2\gamma}\left(\mathbb{T}^{n}\right)\right)$ 
(see \cite[Example 2.12]{GHN21}).\\
\indent
Consider the following concrete semilinear rough SPDE
\begin{equation} \label{rspdec}
	\left\{
		\begin{aligned}
			&du_{t}\left(x\right)=\left[\nabla \cdot \left(a_{t}\left(x\right)\nabla u_{t}\left(x\right)\right)+f\left(t, x, u_{t}\left(x\right), \nabla u_{t}\left(x\right)\right)\right]dt+\left[G_{t}\left(x\right)u_{t}\left(x\right)+g_{t}\left(x\right)\right]d\mathbf{X}_{t}\\
			&\quad \quad \quad \quad+h\left(t, x, u_{t}\left(x\right)\right)dW_{t}, \quad \left(t, x\right) \in \left(0, T\right] \times \mathbb{T}^{n},\\
			&u_{0}\left(x\right)=\xi\left(x\right), \quad x \in \mathbb{T}^{n}.
		\end{aligned}
	\right.	
\end{equation}
Here, $\mathbf{X}=\left(X, \mathbb{X}\right) \in \mathscr{C}^{\alpha}$, 
$\xi: \Omega \times \mathbb{T}^{n} \rightarrow \mathbb{R}$ is the initial datum, 
$f: \left[0, T\right] \times \Omega \times \mathbb{T}^{n} \times \mathbb{R} \times \mathbb{R}^{n} \rightarrow \mathbb{R}$, 
$h: \left[0, T\right] \times \Omega \times \mathbb{T}^{n} \times \mathbb{R} \rightarrow \mathbb{R}^{d}$, 
$G: \left[0, T\right] \times \Omega \times \mathbb{T}^{n} \rightarrow \mathbb{R}^{e}$, 
$G^{\prime}: \left[0, T\right] \times \Omega \times \mathbb{T}^{n} \rightarrow \mathbb{R}^{e \times e}$ 
(implied in \eqref{rspdec}), 
$g: \left[0, T\right] \times \Omega \times \mathbb{T}^{n} \rightarrow \mathbb{R}^{e}$ and 
$g^{\prime}: \left[0, T\right] \times \Omega \times \mathbb{T}^{n} \rightarrow \mathbb{R}^{e \times e}$ 
(also implied in \eqref{rspdec}) are progressively measurable vector fields. 
Given $0 < \beta^{\prime} < \beta < \alpha$ and $m \in \left[2, \infty\right)$ such that $\alpha+\beta+\beta^{\prime} > 1$, 
we introduce the following definition of mild solutions as Definitions \ref{msd} for $\gamma=0$. 
\begin{definition} \label{msdc}
	We call $\left(u, u^{\prime}\right) \in \mathbf{D}_{X}^{\beta, \beta^{\prime}}L_{m}L^{2}\left(\mathbb{T}^{n}\right)$ 
	a mild solution of \eqref{rspdec} if $u^{\prime}=Gu+g$, 
	$\left(Gu+g, Gu^{\prime}+G^{\prime}u+g^{\prime}\right) \in \mathbf{D}_{X}^{\beta, \beta^{\prime}}L_{m}L^{2}\left(\mathbb{T}^{n}, \mathbb{R}^{e}\right)$, 
	and for every $t \in \left[0, T\right]$ and a.s. $\omega \in \Omega$, 
	\begin{equation*}
		u_{t}=S_{0, t}\xi+\int_{0}^{t}S_{r, t}f\left(r, \cdot, u_{r}, \nabla u_{r}\right)dr+\int_{0}^{t}S_{r, t}\left(G_{r}u_{r}+g_{r}\right)d\mathbf{X}_{r}+\int_{0}^{t}S_{r, t}h\left(r, \cdot, u_{r}\right)dW_{r}
	\end{equation*}
	holds in $L^{2}\left(\mathbb{T}^{n}\right)$. 
\end{definition}
For a Banach space $V$, denote by $\mathbf{D}_{X}^{\beta, \beta^{\prime}}L_{\infty}\left(\left[0, T\right], \mathbf{\Omega}, \left\{\mathcal{F}_{t}\right\}; V\right)$ 
the space of stochastic controlled rough paths of $\infty$-integrability and 
$\left(\beta, \beta^{\prime}\right)$-Hölder regularity with values in $V$ 
introduced by Friz et al. \cite[Definition 3.2]{FHL21}. 
We introduce the following assumption. 
\begin{assumption} \label{asc}
	~
	\begin{enumerate}[(i)]
		\item $\mathbf{X}=\left(X, \mathbb{X}\right) \in \mathscr{C}^{\alpha}\left(\left[0, T\right], \mathbb{R}^{e}\right)$; 
		\item $\xi \in L^{m}\left(\Omega, \mathcal{F}_{0}, L^{2}\left(\mathbb{T}^{n}\right)\right)$; 
		\item $t \mapsto f\left(t, \cdot, 0, 0\right)$ is bounded from $\left[0, T\right]$ to $L^{m}\left(\Omega, H^{-1}\left(\mathbb{T}^{n}\right)\right)$ 
		and for every $\left(t, x\right) \in \left[0, T\right] \times \mathbb{T}^{n}$, $u, \bar{u} \in \mathbb{R}$ and $v, \bar{v} \in \mathbb{R}^{n}$, 
		\begin{equation*}
			\left|f\left(t, x, u, v\right)-f\left(t, x, \bar{u}, \bar{v}\right)\right| \lesssim \left|u-\bar{u}\right|+\left|v-\bar{v}\right|;
		\end{equation*}
		\item $t \mapsto h\left(t, \cdot, 0\right)$ is bounded from $\left[0, T\right]$ to $L^{m}\left(\Omega, L^{2}\left(\mathbb{T}^{n}, \mathbb{R}^{d}\right)\right)$ and 
		\begin{equation*}
			\left|h\left(t, x, u\right)-h\left(t, x, \bar{u}\right)\right| \lesssim \left|u-\bar{u}\right|, \quad \forall \left(t, x\right) \in \left[0, T\right] \times \mathbb{T}^{n}, \quad \forall u, \bar{u} \in \mathbb{R}; 
		\end{equation*}
		\item $\left(G, G^{\prime}\right) \in \mathbf{D}_{X}^{\beta, \beta^{\prime}}L_{\infty}\left(\left[0, T\right], \mathbf{\Omega}, \left\{\mathcal{F}_{t}\right\}; L^{\infty}\left(\mathbb{T}^{n}, \mathbb{R}^{e}\right)\right)$ 
		and $\left\|G_{0}\right\|_{\infty, L^{\infty}\left(\mathbb{T}^{n}, \mathbb{R}^{e}\right)} < \infty$; 
		\item $\left(g, g^{\prime}\right) \in \mathbf{D}_{X}^{\beta, \beta^{\prime}}L_{m}L^{2}\left(\mathbb{T}^{n}, \mathbb{R}^{e}\right)$. 
	\end{enumerate}
\end{assumption}
Indeed, Assumption \ref{asc} implies Assumption \ref{as} for $\gamma=\mu=\nu=0$ and $\lambda=\frac{1}{2}$. 
By Theorems \ref{seu} and \ref{csm} and Proposition \ref{sr}, we have 
the following result. 
\begin{theorem} \label{ms}
	Under Assumption \ref{asc}, the concrete rough SPDE \eqref{rspdec} has a unique mild solution 
	$\left(u, u^{\prime}\right) \in \mathbf{D}_{X}^{\beta, \beta^{\prime}}L_{m}L^{2}\left(\mathbb{T}^{n}\right)$ 
	and the mild solution map $\left(\mathbf{X}, \xi, G, G^{\prime}, g, g^{\prime}\right) \mapsto \left(u, u^{\prime}\right)$ 
	is continuous. 
	Furthermore, $u_{t} \in L^{m}\left(\Omega, H^{2\theta}\left(\mathbb{T}^{n}\right)\right)$ 
	for every $t \in \left(0, T\right]$ and $\theta \in \left[0, \frac{1}{2}\right)$. 
\end{theorem}

\bibliographystyle{amsplain}

\begin{biblist}
\bib{ALT22}{article}{
	author={Addona, D.},
	author={Lorenzi, L.},
	author={Tessitore, G.},
	title={Regularity results for nonlinear Young equations and applications},
	journal={J. Evol. Equ.},
	volume={22},
	date={2022},
	number={1},
	pages={Paper No. 3, 34},
	issn={1424-3199},
	review={\MR{4386002}},
	doi={10.1007/s00028-022-00757-y},
}
\bib{ALT24}{article}{
   author={Addona, D.},
   author={Lorenzi, L.},
   author={Tessitore, G.},
   title={Young equations with singularities},
   journal={Nonlinear Anal.},
   volume={238},
   date={2024},
   pages={Paper No. 113401, 33},
   issn={0362-546X},
   review={\MR{4654792}},
   doi={10.1016/j.na.2023.113401},
}
\bib{CDFO13}{article}{
   author={Crisan, D.},
   author={Diehl, J.},
   author={Friz, P. K.},
   author={Oberhauser, H.},
   title={Robust filtering: correlated noise and multidimensional observation},
   journal={Ann. Appl. Probab.},
   volume={23},
   date={2013},
   number={5},
   pages={2139--2160},
   issn={1050-5164},
   review={\MR{3134732}},
   doi={10.1214/12-AAP896},
}
\bib{DGT12}{article}{
   author={Deya, A.},
   author={Gubinelli, M.},
   author={Tindel, S.},
   title={Non-linear rough heat equations},
   journal={Probab. Theory Related Fields},
   volume={153},
   date={2012},
   number={1-2},
   pages={97--147},
   issn={0178-8051},
   review={\MR{2925571}},
   doi={10.1007/s00440-011-0341-z},
}
\bib{DF12}{article}{
   author={Diehl, J.},
   author={Friz, P. K.},
   title={Backward stochastic differential equations with rough drivers},
   journal={Ann. Probab.},
   volume={40},
   date={2012},
   number={4},
   pages={1715--1758},
   issn={0091-1798},
   review={\MR{2978136}},
   doi={10.1214/11-AOP660},
}
\bib{DOR15}{article}{
   author={Diehl, J.},
   author={Oberhauser, H.},
   author={Riedel, S.},
   title={A L\'{e}vy area between Brownian motion and rough paths with
   applications to robust nonlinear filtering and rough partial differential
   equations},
   journal={Stochastic Process. Appl.},
   volume={125},
   date={2015},
   number={1},
   pages={161--181},
   issn={0304-4149},
   review={\MR{3274695}},
   doi={10.1016/j.spa.2014.08.005},
}
\bib{FHL21}{article}{
	author={Friz, P. K.},
	author={Hocquet, A.},
	author={L\^{e}, K.},
	title={Rough stochastic differential equations},
	date={2021},
	eprint={arXiv:2106.10340},
}
\bib{GH19}{article}{
   author={Gerasimovi\v{c}s, A.},
   author={Hairer, M.},
   title={H\"{o}rmander's theorem for semilinear SPDEs},
   journal={Electron. J. Probab.},
   volume={24},
   date={2019},
   pages={Paper No. 132, 56},
   review={\MR{4040992}},
   doi={10.1214/19-ejp387},
}
\bib{GHN21}{article}{
   author={Gerasimovi\v{c}s, A.},
   author={Hocquet, A.},
   author={Nilssen, T.},
   title={Non-autonomous rough semilinear PDEs and the multiplicative sewing
   lemma},
   journal={J. Funct. Anal.},
   volume={281},
   date={2021},
   number={10},
   pages={Paper No. 109200, 65},
   issn={0022-1236},
   review={\MR{4299812}},
   doi={10.1016/j.jfa.2021.109200},
}
\bib{Gub04}{article}{
	author={Gubinelli, M.},
	title={Controlling rough paths},
	journal={J. Funct. Anal.},
	volume={216},
	date={2004},
	number={1},
	pages={86--140},
	issn={0022-1236},
	review={\MR{2091358}},
	doi={10.1016/j.jfa.2004.01.002},
}
\bib{GLT06}{article}{
   author={Gubinelli, M.},
   author={Lejay, A.},
   author={Tindel, S.},
   title={Young integrals and SPDEs},
   journal={Potential Anal.},
   volume={25},
   date={2006},
   number={4},
   pages={307--326},
   issn={0926-2601},
   review={\MR{2255351}},
   doi={10.1007/s11118-006-9013-5},
}
\bib{GT10}{article}{
   author={Gubinelli, M.},
   author={Tindel, S.},
   title={Rough evolution equations},
   journal={Ann. Probab.},
   volume={38},
   date={2010},
   number={1},
   pages={1--75},
   issn={0091-1798},
   review={\MR{2599193}},
   doi={10.1214/08-AOP437},
}
\bib{HN19}{article}{
   author={Hesse, R.},
   author={Neam\c{t}u, A.},
   title={Local mild solutions for rough stochastic partial differential
   equations},
   journal={J. Differential Equations},
   volume={267},
   date={2019},
   number={11},
   pages={6480--6538},
   issn={0022-0396},
   review={\MR{4001062}},
   doi={10.1016/j.jde.2019.06.026},
}
\bib{HN20}{article}{
   author={Hesse, R.},
   author={Neam\c{t}u, A.},
   title={Global solutions and random dynamical systems for rough evolution
   equations},
   journal={Discrete Contin. Dyn. Syst. Ser. B},
   volume={25},
   date={2020},
   number={7},
   pages={2723--2748},
   issn={1531-3492},
   review={\MR{4097587}},
   doi={10.3934/dcdsb.2020029},
}
\bib{HesN22}{article}{
   author={Hesse, R.},
   author={Neam\c{t}u, A.},
   title={Global solutions for semilinear rough partial differential
   equations},
   journal={Stoch. Dyn.},
   volume={22},
   date={2022},
   number={2},
   pages={Paper No. 2240011, 18},
   issn={0219-4937},
   review={\MR{4431448}},
   doi={10.1142/S0219493722400111},
}
\bib{HocN22}{article}{
	author={Hocquet, A.},
	author={Neam\c{t}u, A.},
	title={Quasilinear rough evolution equations},
	date={2022},
	eprint={arXiv:2207.04787},
}
\bib{Le20}{article}{
   author={L\^{e}, K.},
   title={A stochastic sewing lemma and applications},
   journal={Electron. J. Probab.},
   volume={25},
   date={2020},
   pages={Paper No. 38, 55},
   review={\MR{4089788}},
   doi={10.1214/20-ejp442},
}
\bib{Le23}{article}{
   author={L\^{e}, K.},
   title={Stochastic sewing in Banach spaces},
   journal={Electron. J. Probab.},
   volume={28},
   date={2023},
   pages={Paper No. 26, 22},
   review={\MR{4546635}},
   doi={10.1214/23-ejp918},
}
\bib{LS22}{article}{
   author={Li, X.},
   author={Sieber, J.},
   title={Mild stochastic sewing lemma, SPDE in random environment, and
   fractional averaging},
   journal={Stoch. Dyn.},
   volume={22},
   date={2022},
   number={7},
   pages={Paper No. 2240025, 47},
   issn={0219-4937},
   review={\MR{4534216}},
   doi={10.1142/S0219493722400251},
}
\bib{LT23-1}{article}{
	author={Liang, J.},
	author={Tang, S.},
	title={Multidimensional Backward Stochastic Differential Equations with Rough Drifts},
	date={2023},
	eprint={arXiv:2301.12434},
}
\bib{LT23-2}{article}{
	author={Liang, J.},
	author={Tang, S.},
	title={Mild Solution of Semilinear SPDEs with Young Drifts},
	date={2023},
	eprint={arXiv:2309.06791},
}
\bib{Lun18}{book}{
   author={Lunardi, A.},
   title={Interpolation theory},
   series={Appunti. Scuola Normale Superiore di Pisa (Nuova Serie) [Lecture
   Notes. Scuola Normale Superiore di Pisa (New Series)]},
   volume={16},
   edition={3},
   publisher={Edizioni della Normale, Pisa},
   date={2018},
   pages={xiv+199},
   isbn={978-88-7642-639-1},
   isbn={978-88-7642-638-4},
   review={\MR{3753604}},
   doi={10.1007/978-88-7642-638-4},
}
\bib{Lyo94}{article}{
	author={Lyons, T. J.},
	title={Differential equations driven by rough signals. I. An extension of
	an inequality of L. C. Young},
	journal={Math. Res. Lett.},
	volume={1},
	date={1994},
	number={4},
	pages={451--464},
	issn={1073-2780},
	review={\MR{1302388}},
	doi={10.4310/MRL.1994.v1.n4.a5},
}
\bib{Lyo98}{article}{
	author={Lyons, T. J.},
	title={Differential equations driven by rough signals},
	journal={Rev. Mat. Iberoamericana},
	volume={14},
	date={1998},
	number={2},
	pages={215--310},
	issn={0213-2230},
	review={\MR{1654527}},
	doi={10.4171/RMI/240},
}
\bib{Yag10}{book}{
   author={Yagi, A.},
   title={Abstract parabolic evolution equations and their applications},
   series={Springer Monographs in Mathematics},
   publisher={Springer-Verlag, Berlin},
   date={2010},
   pages={xviii+581},
   isbn={978-3-642-04630-8},
   review={\MR{2573296}},
   doi={10.1007/978-3-642-04631-5},
}
\bib{You36}{article}{
	author={Young, L. C.},
	title={An inequality of Hölder type, connected with Stieljes integration},
	journal={Acta Math.},
	volume={67},
	date={1936},
	pages={251--282},
}
\end{biblist}
\bibsection

\end{document}